\documentclass{amsart}
\usepackage[margin=3.9cm]{geometry}

\usepackage{amssymb}
\usepackage{graphicx} 
\usepackage{tikz}
\usepackage{mathrsfs}
\usepackage{enumerate}
\usepackage{xspace}
\usepackage{color}
\usepackage{enumerate}
\usepackage{enumitem}
\usepackage{amsthm}
\usepackage{dsfont}
\usepackage{bbm}
\usepackage[colorlinks=true, linkcolor = blue, citecolor = blue]{hyperref}
\usepackage[numbers]{natbib}
\usepackage[backgroundcolor=white,bordercolor=red]{todonotes}
\DeclareMathAlphabet{\mathpzc}{OT1}{pzc}{m}{it}
\usepackage[nameinlink]{cleveref}
\numberwithin{equation}{section}
\begin{document}

\expandafter\let\expandafter\oldproof\csname\string\proof\endcsname
\let\oldendproof\endproof
\renewenvironment{proof}[1][\proofname]{%
	\oldproof[\scshape\hspace{1em}#1]%
}{\oldendproof}

\newtheoremstyle{mystyle_thm}
{6pt}
{6pt}
{\itshape}
{1em}
{\scshape}
{.}
{.5em}
{}%

\newtheoremstyle{mystyle_def}
{6pt}
{6pt}
{}
{1em}
{\scshape}
{.}
{.5em}
{}%

\theoremstyle{mystyle_thm}

\newtheorem{theorem}{Theorem}[section]
\newtheorem{lemma}[theorem]{Lemma}
\newtheorem{proposition}[theorem]{Proposition}
\newtheorem{corollary}[theorem]{Corollary}
\newtheorem{definition}[theorem]{Definition}
\newtheorem{Ass}[theorem]{Assumption}
\newtheorem{condition}[theorem]{Condition}

\theoremstyle{mystyle_def}

\newtheorem{example}[theorem]{Example}
\newtheorem{remark}[theorem]{Remark}
\newtheorem{SA}[theorem]{Standing Assumption}
\newtheorem{discussion}[theorem]{Discussion}
\newtheorem{remarks}[theorem]{Remark}
\newtheorem*{notation}{Remark on Notation}
\newtheorem{application}[theorem]{Application}
\newtheorem{numexperiment}[theorem]{Numerical Experiment}

\newcommand{\of}{[\hspace{-0.06cm}[}
\newcommand{\gs}{]\hspace{-0.06cm}]}

\newcommand\llambda{{\mathchoice
		{\lambda\mkern-4.5mu{\raisebox{.4ex}{\scriptsize$\backslash$}}}
		{\lambda\mkern-4.83mu{\raisebox{.4ex}{\scriptsize$\backslash$}}}
		{\lambda\mkern-4.5mu{\raisebox{.2ex}{\footnotesize$\scriptscriptstyle\backslash$}}}
		{\lambda\mkern-5.0mu{\raisebox{.2ex}{\tiny$\scriptscriptstyle\backslash$}}}}}

\newcommand{\1}{\mathds{1}}

\newcommand{\F}{\mathbf{F}}
\newcommand{\G}{\mathbf{G}}

\newcommand{\B}{\mathbf{B}}

\newcommand{\M}{\mathcal{M}}

\newcommand{\la}{\langle}
\newcommand{\ra}{\rangle}

\newcommand{\lle}{\langle\hspace{-0.085cm}\langle}
\newcommand{\rre}{\rangle\hspace{-0.085cm}\rangle}
\newcommand{\blle}{\Big\langle\hspace{-0.155cm}\Big\langle}
\newcommand{\brre}{\Big\rangle\hspace{-0.155cm}\Big\rangle}

\newcommand{\X}{\mathsf{X}}

\newcommand{\tr}{\operatorname{tr}}
\newcommand{\N}{{\mathbb{N}}}
\newcommand{\cadlag}{c\`adl\`ag }
\newcommand{\on}{\operatorname}
\newcommand{\oP}{\overline{P}}
\newcommand{\oO}{\mathcal{O}}
\newcommand{\D}{\mathsf{D}} 
\newcommand{\bx}{\mathsf{x}}
\newcommand{\bb}{\hat{b}}
\newcommand{\bs}{\hat{\sigma}}
\newcommand{\bv}{\hat{v}}
\renewcommand{\v}{\mathfrak{m}}
\newcommand{\ob}{\bar{b}}
\newcommand{\oa}{\bar{a}}
\newcommand{\os}{\widehat{\sigma}}
\renewcommand{\j}{\varkappa}
\newcommand{\scl}{\ell}
\newcommand{\Y}{\mathscr{Y}}
\newcommand{\Z}{\mathscr{Z}}
\newcommand{\T}{\mathcal{T}}
\newcommand{\con}{\mathsf{c}}
\newcommand{\nk}{\hspace{-0.25cm}{{\phantom A}_k^n}}
\newcommand{\nl}{\hspace{-0.25cm}{{\phantom A}_1^n}}
\newcommand{\nm}{\hspace{-0.25cm}{{\phantom A}_2^n}}
\newcommand{\n}{\hspace{-0.35cm}{\phantom {Y_s}}^n}
\newcommand{\nme}{\hspace{-0.35cm}{\phantom {Y_s}}^{n - 1}}
\renewcommand{\o}{\hspace{-0.35cm}\phantom {Y_s}^0}
\newcommand{\e}{\hspace{-0.4cm}\phantom {U_s}^1}
\newcommand{\z}{\hspace{-0.4cm}\phantom {U_s}^2}
\newcommand{\iii}{|\hspace{-0.05cm}|\hspace{-0.05cm}|}
\newcommand{\co}{\overline{\on{co}}}
\renewcommand{\k}{\mathsf{k}}
\newcommand{\ovb}{\overline{b}}
\newcommand{\ova}{\overline{a}}
\newcommand{\s}{\mathfrak{s}}
\newcommand{\opsi}{\overline{\Psi}}
\newcommand{\ol}{\mathcal{L}}
\newcommand{\cW}{\mathscr{W}^r}
\newcommand{\MC}{\mathcal{E}^\textup{mc}}
\newcommand{\ccP}{\mathscr{D}} 
\renewcommand{\epsilon}{\varepsilon}
\renewcommand{\rho}{\varrho}

\newcommand{\Cross}{\mathbin{\tikz [x=2ex,y=2ex,line width=.22ex] \draw (0,0) -- (1,1) (0,1) -- (1,0);}}

\newcommand{\fPs}{\fP_{\textup{sem}}}
\newcommand{\fPas}{\mathfrak{S}_{\textup{ac}}}
\newcommand{\rrarrow}{\twoheadrightarrow}
\newcommand{\cA}{\mathcal{A}}
\newcommand{\ocA}{\mathcal{U}}
\newcommand{\cR}{\mathscr{R}}
\newcommand{\cK}{\mathscr{K}}
\newcommand{\cQ}{\mathcal{Q}}
\newcommand{\cF}{\mathcal{F}}
\newcommand{\cE}{\mathcal{E}}
\newcommand{\cC}{\mathscr{D}_{\ell}} 
\newcommand{\ccC}{\mathscr{D}}
\newcommand{\cD}{\mathscr{C}} 
\newcommand{\bC}{\mathbb{C}}
\newcommand{\cH}{\mathcal{H}}
\newcommand{\bth}{\overset{\leftarrow}\theta}
\renewcommand{\th}{\theta}
\newcommand{\cG}{\mathcal{G}}
\newcommand{\fPasn}{\mathfrak{S}^{\textup{ac}, n}_{\textup{sem}}}
\newcommand{\CLM}{\mathfrak{M}^\textup{ac}_\textup{loc}}
\newcommand{\Sd}{\mathcal{S}^\textup{sp}_{\textup{d}}}
\newcommand{\Sc}{\mathcal{S}}
\newcommand{\Sac}{\mathcal{S}_\textup{ac}}
\newcommand{\A}{\mathsf{A}}
\newcommand{\Td}{\mathsf{T}^\textup{d}}
\renewcommand{\t}{\mathfrak{t}}

\newcommand{\bR}{\mathbb{R}}
\newcommand{\nnabla}{\nabla}
\newcommand{\f}{\mathfrak{f}}
\newcommand{\g}{\mathfrak{g}}
\newcommand{\oconv}{\overline{\on{co}}\hspace{0.075cm}}
\renewcommand{\a}{\mathfrak{a}}
\renewcommand{\b}{\mathfrak{b}}
\renewcommand{\d}{\mathsf{d}}
\newcommand{\bS}{\mathbb{S}^d_+}
\newcommand{\p}{\mathsf{p}}
\newcommand{\dr}{r} 
\newcommand{\m}{\mathbb{M}}
\newcommand{\Q}{Q}
\newcommand{\usc}{\textit{USC}}
\newcommand{\lsc}{\textit{LSC}}
\newcommand{\q}{\mathfrak{q}}
\renewcommand{\X}{\mathscr{X}}
\newcommand{\W}{\mathscr{W}}
\newcommand{\fP}{\mathcal{P}}
\newcommand{\w}{\mathsf{w}}
\newcommand{\oM}{\mathsf{M}}
\newcommand{\oZ}{\mathsf{Z}}
\newcommand{\oK}{\mathsf{K}}
\renewcommand{\Re}{\operatorname{Re}}
\newcommand{\cCk}{\mathsf{c}_k}
\newcommand{\C}{\mathsf{C}}
\newcommand{\cP}{\mathcal{P}}
\newcommand{\oPi}{\overline{\Pi}}

\renewcommand{\emptyset}{\varnothing}

\allowdisplaybreaks

\makeatletter
\@namedef{subjclassname@2020}{%
	\textup{2020} Mathematics Subject Classification}
\makeatother

 \title[Robust Market Convergence]{Robust Market Convergence: From Discrete to Continuous Time} 
\author[D. Criens]{David Criens}
\address{Albert-Ludwigs University of Freiburg, Ernst-Zermelo-Str. 1, 79104 Freiburg, Germany}
\email{david.criens@stochastik.uni-freiburg.de}



\thanks{The author 
	thanks Carola Heinzel for invaluable help with python.}
\date{\today}

\maketitle

\begin{abstract}
	Continuous time financial market models are often motivated as scaling limits of discrete time models. The objective of this paper is to establish such a connection for a robust framework. More specifically, we consider discrete time models that are parameterized by Markovian transition kernels, and a continuous time framework with drift and volatility uncertainty, again parameterized in a Markovian way. Our main result is a limit theory that establishes convergence of the uncertainty sets in the Hausdorff metric topology and weak convergence of the associated worst-case expectations. Furthermore, we discuss a structure preservation property of certain approximations. Namely, we establish the convergence of discrete to continuous time robust superhedging prices for some complete robust market models. As illustration of our main results, we use the idea of Kushner's Markov chain approximation method and provide a recursive algorithm for the computation of continuous time robust superhedging prices. 
\end{abstract}

\section{Introduction}
The basic idea of robust finance is to account for model uncertainty by taking not only one but a variety of possible market models into consideration. 
Two prominent frameworks are discrete time setups related to the model of Bouchard and Nutz \cite{BN15} and continuous time models in the spirit of Peng's celebrated \(G\)-expectation \cite{peng2007g}. 
There are many reasons for using either a discrete or a continuous time model. On the one hand, discrete time models are numerically tractable and they appear natural as financial trading takes place on a discrete time grid. On the other hand, continuous time models are very popular as they benefit from a well-developed stochastic calculus and related PDE techniques. 

It is a fundamental question how discrete and continuous time frameworks are related. In particular, convergence results provide a theoretical justification for the general use of continuous models. 

For the single prior case, the relation between discrete and continuous models is well-understood (cf., e.g., \cite{DP92}, \cite[Section~6.4]{BK} and the references therein). 
On the contrary, a systematic study for robust market models seems to be missing. 
This is the objective of this paper. More specifically, we investigate the relation of Markovian versions of the Bouchard--Nutz framework and diffusion-type extensions of Peng's \(G\)-expectation.

Let us explain the models under consideration in more detail. Robust market models are defined via so-called {\em uncertainty sets} consisting of probability measures on suitable path spaces.
For our continuous time setup, the uncertainty set is given by
\begin{equation} \label{eq: D}
	\begin{split}
		\cD = \Big\{ \, &\text{laws of all continuous semimartingales \(X\) whose} \\&\qquad \text{ characteristics \((B, C)\) satisfies a.e. } (d B_t / dt, dC_t / dt) \in \Theta (X_t) \, \Big\}, 
	\end{split}
\end{equation}
where \(\Theta (x) := \{ (b (\lambda, x), a (\lambda, x)) \colon \lambda \in \Lambda\}\) captures drift and volatility uncertainty. The set-valued map \(x \mapsto \Theta(x)\) can be considered as domain of uncertainty. 
 It is worth noting that the special case where \(\Theta (x)\) is independent of the space variable \(x\) corresponds to Peng's \(G\)-expectation. Similar continuous models were for example studied in \cite{CN22b, CN22a, nutz, NVH}. 

For the discrete time setup, we consider an uncertainty set \(\ccP\) of probability measures \(P\) given by
\begin{align} \label{eq: C}
P \Big( \underset{k = 0}{\overset{\infty} \Cross}  d \omega_k \Big) = P^0 (d \omega^0) \, \prod_{k = 1}^\infty P^k (\omega_0, \dots, \omega_{k - 1}, d \omega_k), 
\end{align} 
where \(P^0\) is an initial distribution and \((\omega_0, \dots, \omega_{k - 1}) \mapsto P^k ( \omega_0, \dots, \omega_{k - 1} , d y)\) is a (universally) measurable selection from the set \(\psi (\omega_{k-1}) := \{ \Pi (\lambda, \omega_{k - 1}, d y) \colon \lambda \in \Lambda\}\) with a probability kernel coefficient \(\Pi\) that captures the uncertainty of the transition probabilities. This can be seen as a Markovian version of the Bouchard--Nutz model (\cite{BN15}).

The input for our analysis is a family \((\ccP^{h})_{h > 0}\) of discrete time models \eqref{eq: C} associated to kernel coefficients \((\Pi_h)_{h > 0}\) and the set \(\cD\) of continuous models from \eqref{eq: D}, which is the potential limiting object. 
As in the classical single prior case, it is convenient to consider the natural linear interpolations for the discrete framework that we denote by~\((\cC^{h})_{h > 0}\).

The main contribution of this paper is a limit theory that establishes the convergence \(\cC^h\to\cD\) of the uncertainty sets in the Hausdorff metric topology. This can be seen as an invariance principle for robust market models that provides a theoretical justification for the use of robust diffusion models. Studying the convergence of uncertainty sets in a hyperspace topology appears to be a new approach that distinguishes from previous ones for \(G\)-expectations, stochastic control and the numerical analysis of nonlinear PDEs, which mainly deal with the approximation of worst-case expectations and value functions. As uncertainty sets are the main objects of interest for robust market models, this reflects the different methodologies from robust finance and optimal control. 
Motivated by the dual representation of robust superhedging prices as worst-case expectations (see \cite{BN15,nutzneufeld13,nutz15}), we also prove that 
\begin{align} \label{eq: convergence worst case}
	\sup_{P \in \cC^h} E^P \big[ \varphi \big] \xrightarrow{\quad} \sup_{P \in \cD} E^P \big[ \varphi \big], \quad h \searrow 0,
\end{align} 
for all bounded \(\cD\)-almost surely continuous test functions \(\varphi\). Further, we discuss ramifications for bounded \(\cD\)-almost surely upper and lower semicontinuous functions \(\varphi\). We view \eqref{eq: convergence worst case} as a weak convergence result for ``generalized'' \(G\)-expectations with state-dependent domains of uncertainty.

The prerequisites of our limit theorems can be viewed as robust versions of the Stroock--Varadhan conditions (\cite[Chapter~11]{SV}) for the convergence of Markov chains to diffusions and in the spirit of
\begin{align*}
	\frac{1}{h} \int_{\|y - x\| \leq 1} (y_i - x_i) \, \Pi_h (\lambda, x, dy) &\xrightarrow{\quad} b^i (\lambda, x),
	\\
	\frac{1}{h} \int_{\|y - x\| \leq 1} (y_i - x_i) (y_j - x_j)\, \Pi_h (\lambda, x, dy) &\xrightarrow{\quad} a^{ij} (\lambda, x),
	\\
	 \frac{1}{h} \, \sup_{\lambda, x} \Pi_h (\lambda, x, \{y \colon \|y - x\| \geq \varepsilon\} ) &\xrightarrow{\quad} 0. \phantom {\int_{\|y - x\| \leq 1}}
\end{align*} 
It is worth to mention that these conditions seem to be new in our robust context.
For sake of clarity, we simplified the precise setting and the assumptions above. Indeed, to get the convergence results, the kernel coefficients \((\Pi_h)_{h > 0}\) have to approximate not only the coefficients \(b\) and \(a\) from \(\cD\) but extensions that also take the implicit random drivers into consideration. 
We refer to our main text for the precise statements.

Given the convergence of the market models, it is an interesting follow up question whether the convergence preserves some structure. For example, in classical single prior frameworks, it is known that the convergence of binomial approximations propagates to the associated fundamental price processes, cf. \cite[Section~6.4]{BK}. 
We investigate a closely related question within our robust framework. 
Namely, for a one-dimensional continuous model with drift and elliptic volatility uncertainty, we consider a robust binomial approximation and establish convergence of the associated robust superhedging prices. To prove this result, we pass to dual representations in the spirit of \cite{BN15,nutzneufeld13,nutz15} and deduce their convergence from \eqref{eq: convergence worst case}.

Adapting the idea behind Kushner's celebrated Markov chain approximation method (\cite{K77, K90, KD}), our convergence results lead to a recursive scheme for the computation of robust superhedging prices associated to the model \(\cD\). Let us shortly detail the main idea. Results in \cite{CN_MAFE,nutzneufeld13,nutz15} show that the robust superhedging price \(\pi (f)\) of a contingent claim with payoff \(f\) has the dual representation \(\sup_{Q \in \mathscr{Q}} E^Q [f]\), where \(\mathscr{Q}\) is defined as \(\cD\) with \(\Theta (x)\) replaced by \(\Theta^* (x) := \{ (0, a (\lambda, x)) \colon \lambda \in \Lambda\}\). 
By virtue of \eqref{eq: convergence worst case}, this continuous worst case expectation can be a approximated by discrete versions which can be computed by dynamic programming algorithms for suitable payoff functions (see \cite{bauerle_rieder,bershre,KD}).

Let us now comment on related literature.
Donsker-type theorems for the \(G\)-expectation were established in the papers \cite{D12,DNS12}. More precisely, it was shown that the \(G\)-expectation can be approximated by discrete time analogues. We remark that dynamic programming algorithms are available for our framework (\cite{bauerle_rieder,bershre, KD}). This is also the case for the approximation sequence from \cite{D12} but, as explained in \cite{D12}, not necessarily for those from~\cite{DNS12}. 

Our financial market models have a close connection to stochastic optimal control. To wit, measurable selection arguments show that they admit control representations with feedback controls, which exclusively depend on the controlled process in a possibly non-Markovian manner. Under convexity assumptions on the domains of uncertainty \(\Theta (x)\) and \(\psi^h (x)\), our robust market models also have weak and relaxed control representations (see \cite[Section 5]{CN23a}) that allow randomized controls. The weak control framework can be seen as a convexification of the feedback control setting and the relaxed framework is a convex- and compactification. In this paper, we assume that \(\Theta (x)\) is convex, meaning that our continuous model \(\cD\) has a weak and relaxed control representation (see \cite[Theorem~5.1]{CN23a}). However, we do not assume that the discrete uncertainty sets \((\ccP^h)_{h > 0}\) are convex, i.e., they might not have weak or relaxed control representations. From a financial point of view, it is very natural to aim for a theory that includes non-convex discrete time models. For example, any reasonable robust binomial or tree model lacks convexity. In this regard, our main result \(\cC^h \to \cD\) cannot always be formulated in the language of weak or relaxed control. We refer to Discussion~\ref{diss: MCAM} below for more comments on the relation of our approach to stochastic control theory.

Motivated by Kushner's approximation method, there are many publications that investigate approximation schemes for value functions in weak and relaxed control frameworks, see, e.g., \cite{K77, K90, KD, Menaldi_89}. 
More recently, \cite{T14} provided a probabilistic interpretation of the numerical Monte Carlo scheme for nonlinear parabolic PDEs from \cite{FTW11} in terms of controlled discrete time semimartingales and the papers \cite{LTT23, PTZ21} study different controlled semimartingale approximation schemes in a relaxed control framework with partial information and expectation constraints, respectively. 
As mentioned above, all of these works aim for an approximation of value functions. 
Technically, we adapt many ideas from these references in our proofs. 

The main idea behind the proof of the Hausdorff metric convergence is to adapt the sequential characterizations of upper and lower hemicontinuity of set-valued maps (\cite[Chapter~17]{charalambos2013infinite}) and to prove related properties. We mention that such properties are also established in \cite{DNS12,K90,LTT23,PTZ21,T14} for their (mostly control) frameworks but they were not connected to a hyperspace topology.

For the upper part, we provide two different proofs that are of independent interest. 
In the first, we adapt techniques developed in~\cite{C23b,CN22a} for a time and path continuous robust setting to our framework that contains discrete time objects. This relies on concepts from set-valued analysis. 
The second proof uses the relaxed control formulation of \(\cD\) that was established in \cite{CN23a}. To compare these proofs, the first one is of interest as it does not rely on a parameterization via an action space \(\Lambda\), which is crucial for the use of control arguments, while the second one appears less technical.

For the lower part, we transfer ideas from the monograph \cite{SV} and the papers \cite{DNS12,LTT23,PTZ21} to our setting.
To wit, we use the fact that, under suitable assumptions, piecewise constant strong control rules with continuous weights are dense in the uncertainty set \(\cD\). For such control rules we construct discrete time approximations. 

Given the properties related to upper and lower hemicontinuity, we establish \eqref{eq: convergence worst case} in the spirit of Berge's maximum theorem. 

Finally, let us also mention a delicate point in the proof for the convergence of the superhedging prices. As explained above, the main idea is to use dual representations. For the continuous setting, we can directly refer to \cite{nutz15} but, as the robust binomial setting is not convex, the duality result from \cite{BN15} does not apply to our discrete setting. This forces us to go into the proof from \cite{BN15} and to adapt it to our setting.

Lastly, we comment on a purely analytic approach.
It is well-known (see, e.g., \cite{CN22b, hol16}) that worst-case expectations related to uncertainty sets of the type \(\cD\) have a close connection to sublinear semigroups. 
A general stability theory for convex monotone semigroups has recently been established in the paper \cite{BKN23}. The theory gives access to discrete approximation schemes, cf. \cite[Section~3.6]{BKN23} for an example in the context of stochastic optimal control. The approach from \cite{BKN23} is based on analytic methods, which contrasts our probabilistic and set-valued point of view.

The paper is organized as follows. 
The coefficients and corresponding assumptions are collected in Section~\ref{sec: setting}. Our discrete and continuous time settings are detailed in Section~\ref{sec: MC Di}. The main results are given in Section~\ref{sec: MR}. The algorithm to compute robust superhedging prices can be found in Section~\ref{sec: app robust SHP} and the convergence of discrete to continuous time superhedging prices is discussed in Section~\ref{sec: app2 conv SHP}. The remainder of the paper is dedicated to the proofs. In Section~\ref{sec: MP relation} we establish martingale and control representations of the uncertainty sets. The main 
proofs are then given in the Sections~\ref{sec: pf} and \ref{sec: pf iii}.

\section{From Discrete to Continuous Time} \label{sec: main1}
This section is organized as follows. First, we introduce the coefficients to build the setup. Second, we define a system of discrete time market models and their proposed continuous time limit. Finally, we formulate and discuss our convergence results.

\subsection{The Coefficients and Assumptions} \label{sec: setting}
In this paper, parameter uncertainty will be captured by a so-called action space \(\Lambda\), which we assume to be a compact convex subset of a strictly convex separable Banach space (such as a separable Hilbert space, for instance). This assumption is mainly technical. In all examples presented below, \(\Lambda\) will be a compact convex subset of \(\bR^k\) for some \(k \geq 1\).

For a fixed dimension \(d \geq 2\), \(h > 0\) and \(\lambda \in \Lambda\), let \(\Pi_h (\lambda, \, \cdot \,, \cdot \, )\) be a probability transition kernel from \((\bR^{d}, \mathcal{B}(\bR^{d}))\) into itself.
This coefficient will model the transition probabilities of the approximating sequence.

For a Polish space \(E\), \(\cP (E)\) denotes the space of probability measures on \((E, \mathcal{B}(E))\) endowed with the weak topology \(\sigma (\mathcal{P}(E), C_b (E; \bR))\) of convergence in distribution. 

\begin{SA} \label{SA: 1}
    The map \((\lambda, x) \mapsto \Pi_h (\lambda, x, dy)\) is continuous from \(\Lambda \times \bR^{d}\) into \(\mathcal{P}(\bR^{d})\).
\end{SA}

Next, we introduce the coefficients of the limiting diffusion. They arise from the approximate drift and volatility coefficients that are given by
\begin{align*}
	b^i_h (\lambda, x) &:=	\frac{1}{h} \int_{\|y - x\| \leq 1} (y_i - x_i) \, \Pi_h (\lambda, x, dy), 
	\\
	a_h^{ij} (\lambda, x) &:= \frac{1}{h} \int_{\|y - x\| \leq 1} (y_i - x_i) (y_j - x_j)\, \Pi_h (\lambda, x, dy)
\end{align*}
for \(i, j = 1, \dots, d, \lambda \in \Lambda\) and \(x \in \bR^{d}\).
\begin{SA}\label{SA: 2}
	There exists a constant \(\delta> 0\) such that 
	\begin{align*}
		\sup \Big\{ \| b_h (\lambda, x) \| + \|a_h (\lambda, x) \| \colon h \in (0, \delta],\, \lambda \in \Lambda, \, x \in \bR^{d} \Big\} < \infty, 
	\end{align*}
	and, for every \(\varepsilon > 0\), as \(h \searrow 0\), 
	\begin{align*}
		\Delta_h^\varepsilon := \frac{1}{h} \, \sup \Big\{ \Pi_h (\lambda, x, \{y \colon \|y - x\| \geq \varepsilon\} ) \colon (\lambda, x) \in \Lambda \times \bR^{d} \Big\} \to 0.
	\end{align*}
\end{SA}

\begin{SA} \label{SA: 3}
	There are dimensions \(m, r \in \N\) such that \(m + r = d\) and functions \[b \colon \Lambda \times \bR^{m} \to \bR^{m}, \quad \sigma \colon \Lambda \times \bR^m \to \bR^{m \times r}\] with the following properties:
	\begin{enumerate}
		\item[\textup{(i)}] The coefficients \(b\) and \(\sigma\) are bounded and continuous.
		\item[\textup{(ii)}] For every \(x \in \bR^m\), the set \(\{ (b (\lambda, x), \sigma \sigma^* (\lambda, x)) \colon \lambda \in \Lambda \} \subset \bR^m \times \mathbb{S}^m\) is convex. Here, \(\sigma^*\) denotes the adjoint of \(\sigma\) and \(\mathbb{S}^m\) denotes the set of symmetric positive semidefinite \(m \times m\)  matrices with real entries.
		\item[\textup{(iii)}] As \(h \searrow 0\),
		\begin{align*}
			b_h (\lambda, x)  &\xrightarrow{\quad} (0_r, b (\lambda, \pi_m (x))) =: \ob (\lambda, x), \\
			a_h (\lambda, x) &\xrightarrow{\quad} \begin{pmatrix} \on{Id}_r & \sigma^* (\lambda, \pi_m (x)) \\ \sigma (\lambda, \pi_m (x)) & \sigma \sigma^* (\lambda, \pi_m (x)) \end{pmatrix} =: \oa (\lambda, x)
		\end{align*}
		uniformly in \(\lambda\) and \(x\) on compact subsets of \(\Lambda \times \bR^{d}\). Here, \(0_r\) denotes the zero vector in \(\bR^r\), \(\on{Id}_r\) denotes the \(r \times r\) unit matrix and  \(\pi_m (x) := (x_{r + 1}, \dots, x_{d})\) for \(x = (x_1, \dots, x_d)\in \bR^d\).
	\end{enumerate}
\end{SA}

The Standing Assumptions~\ref{SA: 1} -- \ref{SA: 3} will be in force throughout this paper. For some of our results, we also impose the following condition. 
\begin{condition} \label{cond: main1}
			The coefficients \(b\) and \(\sigma\) from Standing Assumption~\ref{SA: 3} satisfy a Lipschitz condition uniformly in the \(\Lambda\)-variable, i.e., there exists a constant \(C >0 \) such that 
		\[
		\|b (\lambda, x) - b (\lambda, y)\| + \|\sigma (\lambda, x) - \sigma (\lambda, y)\| \leq C \|x - y\|
		\]
		for all \(\lambda \in \Lambda\) and \(x, y \in \bR^m\).
\end{condition}

Let us shortly comment on the structure of \(\ob\) and \(\oa\). If we consider a controlled SDE of the type
\[
d Z_t = b( \lambda_t, Z_t) dt + \sigma (\lambda_t, Z_t) d W_t, 
\]
then the coefficients \(\ob\) and \(\oa\) correspond to the bi-variate process \((W, Z)\) that captures the controlled process together with its driver.

\subsection{The Discrete and Continuous Time Frameworks} \label{sec: MC Di}
In the following we introduce the models under consideration. Recall from Standing Assumption~\ref{SA: 3} that \(d = m + r\) for \(m, r \in \mathbb{N}\).
We start with the continuous time framework. Define 
\[
\Omega^{m} := \Big\{ \omega \colon \bR_+ \to \bR^{m} \text{ continuous} \Big\} 
\]
and endow this space with the local uniform topology.
The coordinate process on \(\Omega^{m}\) is denoted by \(X\). It is well-known that \(\mathcal{B}(\Omega^{m}) = \sigma (X_t, t\geq 0) =: \cF^{m}\). Further, we introduce the natural filtration \(\cF^{m}_t := \sigma (X_s, s \leq t)\) for \(t \geq 0\). 

Let \(\fPas^{m}\) be the set of all probability measures \(P\) on the setup \((\Omega^{m}, \cF^{m}, (\cF^{m}_t)_{t \geq 0})\) such that the coordinate process \(X\) is a continuous \(P\)-semimartingale whose semimartingale characteristics \((B^P, C^P)\) are absolutely continuous w.r.t. the Lebesgue measure \(\llambda\). 
With \(b\) and \(\sigma\) as in Standing Assumption~\ref{SA: 3}, we set 
\begin{align*}
	\Theta (x) := \Big\{ (b (\lambda, x), \sigma \sigma^* (\lambda, x)) \colon \lambda \in \Lambda \Big\} \subset \bR^{m} \times \mathbb{S}^{m}, \quad x \in \bR^m.
\end{align*}
The uncertainty set for the continuous time market models under consideration is given by 
\begin{align*}
	\cD (x) := \Big\{ P \in \fPas^{m} \colon P (X_0 = x) = 1, \ (\llambda \otimes P)\text{-a.e. } (dB^P/ d \llambda, dC^P / d \llambda) \in \Theta (X) \Big\}
\end{align*}
for an initial value \(x \in \bR^m\). 

\begin{example}[Random \(G\)-expectations] \label{rem: GBM}
	To ease our presentation, we restrict our discussion to the case \(m = 1\). 
Let \(b^*, b_* \colon \bR \to \bR\) and \(a^*, a_* \colon \bR \to (0, \infty)\) be continuous functions such that, for a constant \(C > 0\), 
\[
- C \leq b_* \leq b^* \leq C, \quad \frac{1}{C} \leq a_* \leq a^* \leq C.
\]
For a fixed initial value \(x_0\in \bR\), we consider a robust financial market whose physical measures are given by 
\begin{equation} \label{eq: US G BM} \begin{split}
	\mathscr{V} := \Big\{ P \in \fPas^1 \colon & P (X_0 = x_0) = 1,  (\llambda \otimes P)\text{-a.e. } \\& dB^P/ d \llambda \in [b_* (X), b^*(X)], d C^P / d \llambda \in [a_* (X), a^* (X)] \Big\}.
	\end{split}
\end{equation}
This type of model is closely related to the random \(G\)-expectations as studied in \cite{nutz,NVH} and also to the robust market model considered in~\cite{CN_MAFE}. In particular, when \(b_*, b^*, a_*\) and \(a^*\) are constants, we recover Peng's classical \(G\)-expectation, cf. \cite{DHP11}.

To translate the set \(\mathscr{V}\) into our framework, we take \(\Lambda := [0, 1] \times [0, 1]\) and
\begin{align*}
	b (\lambda_1, \lambda_2, x) &:=  b_* (x) + \lambda_1 \, (b^* (x) - b_* (x)), \\
	\sigma (\lambda_1, \lambda_2, x) &:= \sqrt{ a_* (x) + \lambda_2 \, (a^* (x) - a_*(x)) }.
\end{align*}
In this case, it is easily seen that \(\mathscr{V} = \cD (x_0)\). Another parameterization is discussed in Example~\ref{ex: BinM} below.
\end{example}

Next, we introduce the discrete time framework.
The canonical path space for an \(\bR^d\)-valued discrete time process is given by  
\[
\Sigma^d := \Big\{ f \colon \mathbb{Z}_+ \to \bR^{d} \Big\}.
\] 
We denote the coordinate process on this space by \(Y = (Y_k)_{k = 0}^\infty\). 
Further, we consider the natural \(\sigma\)-field and filtration generated by the coordinate process, i.e., 
\[ 
\cG^d := \sigma (Y_k, k \geq 0), \quad \cG_n^d := \sigma (Y_k, k \leq n ), \quad n \geq 0.
\]

For \(h > 0\) and \(x \in \bR^d\), we set 
\[
\psi^h (x) := \Big\{ \Pi_h (\lambda, x, dy) \colon \lambda \in \Lambda\Big\}.
\]
The set-valued map \(x \mapsto \psi^h (x)\) has compact values and it is continuous, see \cite[Proposition~1.4.14]{AF}. In turn, it admits a measurable selector (\cite[Theorem~12.1.10]{SV}).

For \(h > 0\) and \(x \in \bR^m\), we define \(\ccP^h (x)\) to be the set of all probability measures \(P\) on \((\Sigma^d, \cG^d)\) such that 
\begin{align} \label{eq: P prod}
P \Big( \underset{k = 0}{\overset{\infty} \Cross}  d \omega_k \Big) = P^0 (d \omega^0) \, \prod_{k = 1}^\infty P^k (\omega_0, \dots, \omega_{k - 1}, d \omega_k), 
\end{align}
	where \(P^0 := \delta_{(0_r, x)}\) and \(P^n\) is a measurable selector of \(\psi^h (Y_{n - 1})\). Notice that the set \(\ccP^h (x)\) has the same product structure that is used in \cite{BN15}.
	Evidently, for \(P\) as in \eqref{eq: P prod}, \(P\)-a.s. 
	\[
	P ( Y_k \in dy \mid \cG	^d_{k -1}) = P^k (dy) \in \psi^h (Y_{k - 1}) = \Big\{ \Pi_h (\lambda, Y_{k-1}, dy) \colon \lambda \in \Lambda\Big\}.
	\]	
	This shows that \(P\) is the law of a controlled Markov chain with feedback controls.\footnote{Here, {\em feedback controls} do not refer to controls with a Markovian structure but to dependence solely on the paths of the controlled process.}
	
\begin{example}[Cox--Ross--Rubinstein Model under Parameter Uncertainty] \label{ex: CRR}
	We now present a discrete time model that will turn out to be an approximation of the continuous time random \(G\)-expectation model from Example~\ref{rem: GBM}. 
	Let \(\Lambda := [0, 1] \times [0, 1]\), and take \(b\) and \(\sigma\) as in Example~\ref{rem: GBM}.
	Further, take a random variable \(\xi\) with zero expectation and unit variance, which might have an absolutely continuous or discrete distribution. 
	For \(h > 0\) and \((\lambda, x = (x_1, x_2)) \in \Lambda \times \bR^2\), let \(\Pi_h (\lambda, x, dy)\) be the distribution of the bi-variate random variable 
	\begin{align} \label{ex: 2.6, rade}
	x + \big( \xi \, \sqrt{h}, \, b (\lambda, x_2) \, h + \sigma (\lambda, x_2) \, \xi \, \sqrt{h}\, \big).
	\end{align} 
	It is routine to check that the Standing Assumptions~\ref{SA: 1} -- \ref{SA: 3} are satisfied in this case.

To give an interpretation, suppose that \(b \equiv 0\) and that \(\xi\) has a Rademacher distribution, i.e., \(\xi = \pm 1\) with equal probability. In this case, for each \(h > 0\), in each time step the paths go either up or down with a size that is object to parameter uncertainty (represented by the dependence on the set \(\Lambda\), which may be seen as all possible parameters under consideration). 
\end{example}

\subsection{The Convergence Result} \label{sec: MR}
Using linear interpolation, we connect the approximating discrete time models to the path space of the continuous model. 
More precisely, with
\begin{align} \label{eq: interpol}
\X_t (h) := \Big( \Big\lfloor \, \frac{t}{h}\, \Big \rfloor + 1 - \frac{t}{h} \Big)  \, \pi_m (Y_{\lfloor t/h\rfloor}) + \Big( \frac{t}{h} - \Big\lfloor \, \frac{t}{h}\, \Big \rfloor\Big) \, \pi_m (Y_{\lfloor t/h \rfloor + 1}), \quad t \geq 0,
\end{align}
we set
\[
\cC^h (x) := \Big\{ P \circ \X (h)^{-1} \colon P \in \ccP^h (x) \Big\} \subset \mathcal{P} (\Omega^m), \quad x \in \bR^m.
\]
Here, recall that \(\pi_m (x_1, \dots, x_d) = (x_{r + 1}, \dots, x_d)\).

Let \(\p\) be a metric on \(\mathcal{P}(\Omega^m)\) that induces its topology of convergence in distribution. For two sets \(A, B \subset \cP (\Omega^m)\), we define its Hausdorff distance by
\begin{align*}
	\mathsf{h} (A, B) := \max \Big\{ \sup_{x \in A} \p (x, B), \, \sup_{y \in B} \p (y, A) \Big\}, 
\end{align*}
where \(\p (x, B) := \inf_{y \in B} \p (x, y)\) denotes the distance function. The map \(\mathsf{h}\) is called {\em Hausdorff metric} on the space of all subsets of \(\cP (\Omega^m)\). This name is slightly misleading as \(\mathsf{h}\) is only an extended pseudometric, i.e., a pseudometric that may take the value \(+ \infty\). 
 In general, Hausdorff metric convergence may depend on the metric of the underlying space. In the following theorem, which constitutes the main result of this paper, we present a convergence result in terms of \(\mathsf{h}\) that is independent of the particular choice of \(\mathsf{p}\). 
 
 To fix some terminology, for a set \(\mathscr{P}\) of probability measures, we say that a property holds {\em \(\mathscr{P}\)-almost surely}, written \(\mathscr{P}\)-a.s., if it holds \(P\)-a.s. for every \(P \in \mathscr{P}\). 
 
 We say that a sequence \((\varphi^n)_{n = 1}^\infty\) of real-valued functions on \(\Omega^m\) converges to a function \(\varphi \colon \Omega^m \to \bR\) {\em boundedly and uniformly on compacts (b-uc)} if 
 \[
 \sup_{n \in \mathbb{N}} \| \varphi^n \|_\infty < \infty, 
 \]
 and \(\varphi^n \to \varphi\) uniformly on all compact subsets of \(\Omega^m\). 

\begin{theorem}  \label{theo: main1}
Suppose that the Standing Assumptions~\ref{SA: 1}, \ref{SA: 2} and \ref{SA: 3} are in force.
\begin{enumerate}
	\item[\textup{(i)}] For every \(x \in \bR^m\), the sets \(\cD (x)\) and \(\cC^h (x)\) are nonempty and \(\cD (x)\) is a compact subset of \(\cP(\Omega^m)\).
\end{enumerate}
If Condition~\ref{cond: main1} holds additionally, we also have the following:
\begin{enumerate}
	\item[\textup{(ii)}] For every sequence \((h_n)_{n = 1}^\infty \subset (0, \infty)\) with \(h_n \searrow 0\), and every sequence \((x_n)_{n = 0}^\infty \subset \bR^m\) with \(x_n \to x_0\), 
	\[
	\mathsf{h} (\cC^{h_n} (x_n) , \cD (x_0)) \to 0, \quad n \to \infty, 
	\]
            \item[\textup{(iii)}]
            	and, for every bounded \(\cD (x_0)\)-a.s. continuous function \(\varphi \colon \Omega^{m} \to \bR\), and any sequence \((\varphi^n)_{n = 1}^\infty\) of upper semianalytic functions \(\Omega^m \to \bR\) such that \(\varphi^n \to \varphi\) b-uc, 
	\[
	\sup_{Q \in \cC^{h_n} (x_n)} E^Q \big[ \varphi^n \big] \to \sup_{Q \in \cD (x_0)} E^Q \big[ \varphi \big], \quad n \to \infty.
	\]
\end{enumerate}
\end{theorem}

	\begin{remark}[On time-dependent coefficients]
		Our main Theorem~\ref{theo: main1} is formulated for a time-homogeneous framework. As the diffusion coefficient is not assumed to be non-degenerate, time-dependence can be incorporated into our framework by considering the time-space process (see \cite[Exercise~6.7.2]{SV} for this approach in the uncontrolled setting). More precisely, to treat time-dependent coefficients \(\tilde{b} \colon \Lambda \times \bR_+ \times \bR^m \to \bR^m\) and \(\tilde{\sigma} \colon \Lambda \times \bR_+ \times \bR^m \to \mathbb{R}^{m \times r}\) consider the auxiliary coefficients \(\hat{b} \colon \Lambda \times \bR^{m + 1} \to \bR^{m + 1}\) and \(\hat{\sigma} \colon \Lambda \times \bR^{m + 1} \to \bR^{m + 1 \times r}\) given by   
		\begin{align*}
			\hat{b} (\lambda, x_1, x_2,  \dots, x_{m + 1}) &:= \big(1, \tilde{b}\, (\lambda, x_1 \vee 0, x_2, \dots, x_{m + 1}) \big), \\ 
			\hat{\sigma} (\lambda, x_1, x_2, \dots, x_{m + 1}) &:= 
			\begin{pmatrix} 0  \\  \tilde{\sigma} (\lambda, x_1 \vee 0, x_2, \dots, x_{m + 1}) \end{pmatrix}.
		\end{align*} 
		This approach captivates through its simplicity but it does not lead to optimal assumptions, because Condition~\ref{cond: main1} requires Lipschitz continuity in the time domain. Going through the proofs below, it becomes evident that mere continuity in time is sufficient. We leave the details to the reader.
	\end{remark}

\begin{example}
	Provided the coefficients are Lipschitz continuous, Theorem~\ref{theo: main1} shows that the robust Cox--Ross--Rubinstein model from Example~\ref{ex: CRR} converges to the random \(G\)-expectation model from Example~\ref{rem: GBM}. This resembles the well-known fact that the classical Cox--Ross--Rubinstein model converges to the Bachelier (or Black--Scholes, when formulated via log-returns) model.
\end{example}

\begin{example}[Binomial Approximation of Example~\ref{rem: GBM}] \label{ex: BinM}
	We now discuss a binomial approximation of the market model from Example~\ref{rem: GBM}.
	Let \(b^*, b_* \colon \bR \to \bR\) and \(a^*, a_* \colon \bR \to (0, \infty)\) be Lipschitz continuous functions such that, for a constant \(C > 0\), 
	\[
	- C \leq b_* \leq b^* \leq C, \quad \frac{1}{C} \leq a_* \leq a^* \leq C.
	\]
	We take the action space
	\begin{align*}
		\Lambda := [0, 1] \times [0, 1] \times [0, 1] \times [0, 1].
	\end{align*}
	For \((\lambda = (\lambda_1, \dots, \lambda_4), x = (x_1, x_2)) \in \Lambda \times \bR^2\) and \(h > 0\), define 
	\begin{align*}
		\Phi_h (\lambda, x, dy) :=p (\lambda, x) \, &\delta_{x + ( \sqrt{h} \, v (\lambda, x),\, u_h (\lambda, x))} (dy) \\&\quad+ (1 - p (\lambda, x))\, \delta_{x + (- \sqrt{h} / v (\lambda, x),\, d_h (\lambda, x))} (dy), 
	\end{align*}
	where 
	\begin{align*}
		p (\lambda, x) &:= \frac{\sqrt{a_* (x_2)  + \lambda_4 (a^* (x_2) - a_* (x_2))}}{\sqrt{a_* (x_2)  + \lambda_4 (a^* (x_2) - a_* (x_2))} + \sqrt{a_* (x_2)  + \lambda_3 (a^* (x_2) - a_* (x_2))}}, \\
		v (\lambda, x) & := \sqrt{\frac{\sqrt{a_* (x_2)  + \lambda_3 (a^* (x_2) - a_* (x_2))}}{\sqrt{a_* (x_2)  + \lambda_4 (a^* (x_2) - a_* (x_2))}}},  
		\\
		u_h (\lambda, x) &:=  h\, ( b_* (x_2) + \lambda_1 (b^* (x_2) - b_* (x_2))) + \sqrt{h} \, \sqrt{a_* (x_2)  + \lambda_3 (a^* (x_2) - a_* (x_2))}, \phantom{\int^a}
		\\
		d_h (\lambda, x) &:=  h\, ( b_* (x_2) + \lambda_2 (b^* (x_2) - b_* (x_2))) -  \sqrt{h} \, \sqrt{a_* (x_2) + \lambda_4 (a^* (x_2) - a_* (x_2))}. \phantom{\int^a}
	\end{align*}
	We have 
	\begin{align*}
		\Big\{ (u_h (\lambda, x), d_h (\lambda, x)) \colon \lambda \in \Lambda \Big\} = \Big[ h b_* (x) &+ \sqrt{h a_* (x) }, h b^* (x) + \sqrt{h a^* (x)} \, \Big] 
		\\&\times \Big[ h b_* (x) - \sqrt{h a^* (x) }, h b_* (x) - \sqrt{h a_* (x) }\,\Big].
	\end{align*}
	Putting this in perspective, the above model considers up and down jumps with uncertain jump sizes taken from the above intervals.
	A short computation yields that 
	\begin{align*}
		\frac{1}{h} \int (x_1 - y_1) \, \Pi_h (\lambda, x, dy) &= 0, 
		\\
		\frac{1}{h} \int (x_2 - y_2) \, \Pi_h (\lambda, x, dy) & = p (\lambda, x) ( b_* (x_2) + \lambda_1 (b^* (x_2) - b_* (x_2))) 
		\\&\hspace{0.75cm}+ (1 - p (\lambda ,x))(b_* (x_2) + \lambda_2 (b^* (x_2) - b_* (x_2))),
		\\
		\frac{1}{h}	\int (x_1 - y_1)^2\, \Pi_h (\lambda, x, dy) &= 1,
		\\ 
		\frac{1}{h}	\int (x_1 - y_1)(x_2 - y_2)\, \Pi_h (\lambda, x, dy) &\xrightarrow{h \, \searrow \, 0} \sqrt{w (\lambda, x)}, 
		\\
		\frac{1}{h}	\int (x_2 - y_2)^2\,  \Pi_h (\lambda, x, dy) &\xrightarrow{h \, \searrow \, 0} w (\lambda, x),
	\end{align*}
	where
	\begin{align*}
		w (\lambda, x) :=\sqrt{ (a_* (x_2)  + \lambda_3 (a^* (x_2) - a_* (x_2)))(a_* (x_2)  + \lambda_4 (a^* (x_2) - a_* (x_2)))}.
	\end{align*}
	Therefore, noting that 
	\begin{align*}
		\Big\{ p (\lambda, x) ( &b_* (x_2) + \lambda_1 (b^* (x_2) - b_* (x_2))) 
		\\&+ (1 - p (\lambda , x))(b_* (x_2) + \lambda_2 (b^* (x_2) - b_* (x_2))) \colon \lambda \in \Lambda \Big\} = \big[ b_* (x_2), b^* (x_2) \big], \\
		&\hspace{5.55cm}\Big\{ w (\lambda, x)  \colon \Lambda \in \Lambda \Big\} = \big[ a_* (x_2), a^* (x_2) \big], 
	\end{align*}
	it follows in a straightforward manner from Theorem~\ref{theo: main1} that the above discrete time market model is an approximation of the continuous market model with uncertainty set~\eqref{eq: US G BM}.
\end{example}

Part (i) from Theorem~\ref{theo: main1} follows from \cite[Theorem~4.4]{CN22a} and standard extension arguments.
The proof for part (ii) is given in Section~\ref{sec: pf} below. Part (iii) follows from the following (slightly) stronger results, whose proofs can be found in Section~\ref{sec: pf iii} below. 
\begin{theorem} \label{theo: upper}
	Suppose that the Standing Assumptions~\ref{SA: 1} -- \ref{SA: 3} are in force.
	For every sequence \((x_n)_{n = 0}^\infty \subset \bR^m\) with \(x_n \to x_0\), every bounded \(\cD (x_0)\)-a.s. upper semicontinuous function \(\varphi \colon \Omega^{m} \to \bR\), every sequence \((\varphi^n)_{n = 1}^\infty\) of upper semianalytic functions \(\Omega^m \to \bR\) with \(\varphi^n \to \varphi\) b-uc, and every sequence \((h_n)_{n = 1}^\infty \subset (0, \infty)\) with \(h_n \searrow 0\), 
	\begin{align*}
		\limsup_{n \to \infty} \sup_{P \in \cC^{h_n} (x_n)} E^{P} \big[ \varphi^n \big] \leq \sup_{P \in \cD (x_0)} E^P \big[ \varphi \big].
	\end{align*}
\end{theorem}

\begin{theorem} \label{theo: lower}
		Suppose that the Standing Assumptions~\ref{SA: 1} -- \ref{SA: 3} and Condition~\ref{cond: main1} are in force.
	For every sequence \((x_n)_{n = 0}^\infty \subset \bR^m\) with \(x_n \to x_0\), every bounded \(\cD(x_0)\)-a.s. lower semicontinuous function \(\varphi \colon \Omega^{m} \to \bR\), every sequence \((\varphi^n)_{n = 1}^\infty\) of upper semianalytic functions \(\Omega^m \to \bR\) with \(\varphi^n \to \varphi\) b-uc, and every sequence \((h_n)_{n = 1}^\infty \subset (0, \infty)\) with \(h_n \searrow 0\),
	\begin{align*}
		\liminf_{n \to \infty} \sup_{P \in \cC^{h_n} (x_n)} E^{P} \big[ \varphi^n \big] \geq \sup_{P \in \cD (x_0)} E^P \big[ \varphi \big].
	\end{align*}
\end{theorem}

\begin{remark}
\begin{enumerate}
	\item[\textup{(a)}]
Part (iii) from Theorem~\ref{theo: main1} implies that the map 
\[
x \mapsto \sup_{Q \in \cD (x)} E^Q \big[ \varphi \big] 
\]
is continuous from \(\bR^m\) into \(\bR\). This fact is (essentially) known (cf., e.g., \cite{CN22b,CN22a}) but we provide a new proof.
\item[\textup{(b)}] A class of functions that might only be almost surely continuous (or a.s. semi-continuous), but not continuous everywhere, is given by \(\varphi = g (X_{\tau_D})\), where \(\tau_D := \inf \{t \geq 0 \colon X_t \not \in D\}\) for a domain \(D \subset \bR^m\). In general, \(\tau_D\) is only lower semicontinuous but, under suitable assumptions on the domain \(D\) and the coefficients \(b\) and \(a\), \(\tau_D\) turns out to be \(\cD (x_0)\)-a.s. finite and continuous, which then transfers to \(g (X_{\tau_D})\) under suitable assumptions on the function \(g \colon \bR^m \to \bR\).
\end{enumerate}
\end{remark}

	\begin{discussion} 
		\label{diss: MCAM}
It is interesting to understand how Theorem~\ref{theo: main1} relates to relaxed control rules, which can be seen as a convex- and compactification of our setting. 
 By \cite[Theorem~5.1]{CN23a}, the set \(\cD\) has a relaxed control representation under Standing Assumption~\ref{SA: 3}, see Lemma~\ref{lem: relaxed controls} below for a restatement of this result. This is not necessarily the case for \((\ccP^h)_{h > 0}\), because these sets might not be convex. In this regard, part (ii) from Theorem~\ref{theo: main1} has no direct translation into the language of relaxed controls. 
Nevertheless, combining our methods with ideas from~\cite{K90}, one can prove a version of Theorem~\ref{theo: main1}~(ii) in a relaxed control setting, which appears to be a new perspective for such frameworks.\footnote{We again stress that in relaxed frameworks the uncertainty sets (in particular the analogs of \((\ccP^h)_{h > 0}\)) are necessarily convex.}

Let us also comment on the relation of (ii) and (iii) from Theorem~\ref{theo: main1}. 
We start with a simple fact (see \cite[Lemma~12.1.7]{SV} for a related result). 
\begin{lemma} \label{lem: max continuous}
Let \(E\) be a Polish space, denote its compact subsets by \(\on{comp} (E)\) and endow this space with the Hausdorff metric topology. Further, take a continuous function \(g \colon E \to \bR\). Then, 
	\[
	K \mapsto \max_{x \in K} g (x)
	\] 
	is continuous from \(\on{comp} (E)\) into \(\bR\).
\end{lemma} 
\begin{proof}
	By virtue of \cite[Theorem~3.91]{charalambos2013infinite}, it suffices to understand that \(\{ K \colon \max_{K} g \in A \}\) is closed (resp. open) in case \(A \subset \bR\) is closed (resp. open). Because
	\[
	\Big\{ K \in \on{comp} (E) \colon \max_K g \in A \Big\} = \Big\{ K \in \on{comp} (E)\colon K \subset \{ g \in A \} \Big\}
	\]
	and \(\{g \in A\}\) is open or closed according whether \(A\) is open or closed (because \(g\) is continuous), this follows from \cite[Lemma~12.1.2]{SV}. 
\end{proof}

In case each \(\cC^{h_n} (x_n)\) is compact,\footnote{This holds if (in addition to our other assumptions) the set  \(\psi^h (x) = \{ \Pi_h (\lambda, x, dy) \colon \lambda \in \Lambda\}\) is convex. Indeed, under this hypothesis our setting has a relaxed control formulation that entails compactness.} part (ii) from Theorem~\ref{theo: main1} and Lemma~\ref{lem: max continuous} yield that 
\[
\max_{P \in \cC^{h_n} (x_n)} E^P \big[ \varphi \big] \to \max_{P \in \cD(x_0)} E^P \big[ \varphi \big] 
\]
for every bounded continuous function \(\varphi \colon \Omega^m \to \bR\). This resembles a version of part (iii) from Theorem~\ref{theo: main1}. 

In the relaxed control framework the uncertainty sets are typically compact and, via Lemma~\ref{lem: max continuous}, Hausdorff metric convergence propagates naturally to value functions. This explains how a control version of Theorem~\ref{theo: main1}~(ii) relates naturally to results that are of interest in control theory.
	\end{discussion}

\section{Computation and Convergence of Robust Superhedging Prices}

In this section, we present two applications of Theorem~\ref{theo: main1}. First, we discuss how it yields a simple algorithm to compute robust superhedging prices and second, we prove that the binomial approximation in the spirit of Example~\ref{ex: BinM} propagates to the robust superhedging prices.

\subsection{Computation of Robust Superhedging Prices}  \label{sec: app robust SHP}
Kushner's celebrated Markov chain approximation method (\cite{K77,K90,KD}) is a powerful numerical approach for Hamilton--Jacobi--Bellman PDEs. The basic idea is to consider a control solution in terms of a controlled diffusion, to approximate latter by a controlled Markov chain, to solve the associated discrete time problem and finally, to show that the solutions converge to each other. 
	
	In the following, in terms of a simple recursive schemes for robust superhedging prices, we adapt this idea to our financial context. 
	At this point, we stress that our discussion seeks only for an illustration of the idea and not for a comprehensive numerical study.

\smallskip 

To ease our presentation, we consider the one-dimensional setting from Example~\ref{rem: GBM}. To recall the assumptions, let \(b^*, b_* \colon \bR \to \bR\) and \(a^*, a_* \colon \bR \to (0, \infty)\) be continuous functions such that, for \(C > 0\), 
\[
- C \leq b_* \leq b^* \leq C, \quad \frac{1}{C} \leq a_* \leq a^* \leq C.
\]
The uncertainty set of the model under consideration is given by 
\begin{align*}
		\mathscr{V} := \Big\{ P \in \fPas^1 \colon & P (X_0 = x_0) = 1,  (\llambda \otimes P)\text{-a.e. } \\& dB^P/ d \llambda \in [b_* (X), b^*(X)],\, d C^P / d \llambda \in [a_* (X), a^* (X)] \Big\},
\end{align*}
where \(x_0 \in \bR\) is an arbitrary initial value.
We denote the corresponding set of ``equivalent local martingale measures'' by
\begin{align*}
	\mathscr{M} := \Big\{ Q \in \mathcal{P} (\Omega^1) \colon \exists \, P \in \mathscr{V} \text{ with } P \sim_\textup{loc} Q \text{ and \(X\) is a local \(Q\)-martingale} \Big\}.
\end{align*}
We are interested in computing the corresponding robust superhedging price of a contingent claim with payoff
\begin{align} \label{eq: payoff f}
	f := \int_0^T g (X_s) ds + \ell (X_T),
\end{align}
where \(T \in (0, \infty)\) is a fixed time horizon and \(g, \ell \in C_b (\bR; \bR)\). The integral \(\int_0^T g (X_s) ds\) is a running payoff and \(\ell (X_T)\) is a terminal payoff. We restrict our attention to such payoffs, because they allow us to use classical dynamic programming algorithms.

The robust superhedging price is given by the formula
\begin{align*}
	\pi (f) 
	:= \inf \Big\{ x \in \bR \colon \exists \, H \in \mathcal{H} \text{ with } x + \int_0^T H_s d X_s \geq f \ \mathscr{M}\text{-a.s.} \Big\},
	\end{align*}
where \(\mathcal{H}\) denotes the set of all admissible integrands, i.e., all predictable integrands \(H\) such that \(\int_0^{\cdot \wedge T} H_s d X_s\) is a well-defined as a \(P\)-supermartingale for all \(P \in \mathscr{M}\), cf.~\cite[p.~4550]{nutz15}.

\begin{proposition} \label{prop: computation RSP}
	Assume that \(a_*\) and \(a^*\) are Lipschitz continuous. Then, 
\begin{align} \label{eq: approx formula} 
	\pi (f) = \lim_{h \searrow 0} \,  S_h^{\lfloor T / h \rfloor} (\ell) (x_0),
\end{align} 
where \(S^N_h\) denotes the \(N\)-times composition of the operator
\[
S_h (J) (x) := h g (x) + \frac{1}{2} \, \sup_{\lambda \in [0, 1]} \Big\{ J (x + \sqrt{h}\, \sigma (\lambda, x) ) \, + J (x - \sqrt{h}\, \sigma (\lambda, x)) \Big\}
\]
with \(\sigma (\lambda, x) := \sqrt{ a_* (x) + \lambda\, (a^* (x) - a_* (x))}\).
\end{proposition}

\begin{remark}
	The value
	\(
	S_h^{\lfloor T / h \rfloor} (\ell) (x_0)
	\)
	in the formula \eqref{eq: approx formula} can be computed by backward iteration (see also the algorithm on p.~24 in \cite{bauerle_rieder}). 
	
	The parameterization \(\sigma (\lambda, x) = \sqrt{ a_* (x) + \lambda\, (a^* (x) - a_* (x))}\) in Proposition~\ref{prop: computation RSP} is taken for convenience. It will be clear from the proof that any parameterization \(\sigma (\lambda, x)\) that satisfies Condition~\ref{cond: main1} would do the job.
\end{remark}

\begin{proof}[Proof of Proposition~\ref{prop: computation RSP}]
	The robust superhedging theorem \cite[Theorem~3.2]{nutz15} allows us to pass to the dual formulation of the robust superhedging price that is given by
	\begin{equation} \label{eq: superhedging duality}
		\begin{split}
			\pi (f) = \sup_{Q \in \mathscr{M}} E^Q \big[ f \big].
		\end{split}
	\end{equation}
Next, we translate the r.h.s. into our framework from the previous section. By a minor extension of \cite[Corollary~2.13]{CN_MAFE} to an infinite time horizon setting, we obtain that 
\begin{equation} \label{eq: concrete}
	\begin{split}
		\mathscr{M} = \Big\{ P \in \fPas^1 \colon & P (X_0 = x_0) = 1,  (\llambda \otimes P)\text{-a.e. } \\& dB^P/ d \llambda = 0,\, d C^P / d \llambda \in [a_* (X), a^* (X)] \Big\}.
	\end{split}
\end{equation}
With this representation at hand, we can follow Kushner's approximation method and set up an approximation scheme for the worst case expectation in~\eqref{eq: superhedging duality}. 

We use part (iii) from Theorem~\ref{theo: main1}. In order to do so, we have to translate \eqref{eq: concrete} into the framework from the previous section. 
Similar to Example~\ref{rem: GBM}, take \(\Lambda := [0, 1], b \equiv 0\) and
\begin{align*}
	\sigma (\lambda, x) &:= \sqrt{ a_* (x) + \lambda \, (a^* (x) - a_*(x)) }.
\end{align*}
It is easy to see that \(\mathscr{M} = \cD (x_0)\). 

Next, we introduce the approximation scheme. 
We start with the payoff function \(f\). Set 
\begin{align} \label{eq: payoff approx}
	f^h := \sum_{i = 0}^{\lfloor T / h \rfloor - 1} h g(X_{i h}) + \ell (X_{\lfloor T / h \rfloor h} ), \quad h > 0.
\end{align}
It is evident that \(\|f^h\|_\infty \leq  T\,\|g\|_\infty + \|\ell\|_\infty\). 
Further, the Arzel\`a--Ascoli theorem shows that 
\begin{align*}
	\big| f - f^h \big| \leq \sum_{i = 0}^{\lfloor T / h \rfloor - 1} &\int_{i h}^{(i + 1) h} \big| g (X_s) - g(X_{ih}) \big| ds \\&+ \int_{\lfloor T / h \rfloor h}^T | g (X_s) | ds + \big| \ell (X_T) - \ell (X_{ \lfloor T / h \rfloor h}) \big| \xrightarrow{\quad} 0
\end{align*}
uniformly on compacts as \(h \searrow 0\).
Hence, \(f^h \to f\) b-uc as \(h \searrow 0\). 

As in Example~\ref{ex: CRR}, let \(\Pi_h( \lambda, x, dy)\) be the distribution of \eqref{ex: 2.6, rade} with \(b \equiv 0\) and \(\xi\) such that \(\xi = \pm 1\) with equal probability.
It is straightforward to check that the standing assumptions and Condition~\ref{cond: main1} are satisfied. Thus, we can apply Theorem~\ref{theo: main1} and conclude that 
\begin{align*}
	\sup_{P \in \cC^{h} (x_0)} E^P \big[ f^h (\on{pr}_{\Omega^1}) \big] \to \pi (f), \quad h \searrow 0,
\end{align*}
where \(\on{pr}_{\Omega^1} \colon \Omega^2 \to \Omega^1\) denotes the projection to the second coordinate, i.e., \(\on{pr}_{\Omega^1}  (\omega^1, \omega^2) = \omega^2\) for all \(\omega = (\omega^1, \omega^2) \in \Omega^2\).

It remains to show that this coincides with the formula \eqref{eq: approx formula}.
We deduce from \cite[Theorems~2.2.3, 2.3.7]{bauerle_rieder} (or see \cite[Propositions~8.1, 8.2]{bershre}) that  
\begin{align*}
	\sup_{P \in \cC^{h} (x_0)} E^P \big[ f^h (\on{pr}_{\Omega^1}) \big] = S_h^{\lfloor T / h \rfloor} (\ell) (x_0). 
\end{align*}
This completes the proof.
\end{proof}
	\begin{figure}
				\centering
		\includegraphics[scale=.45]{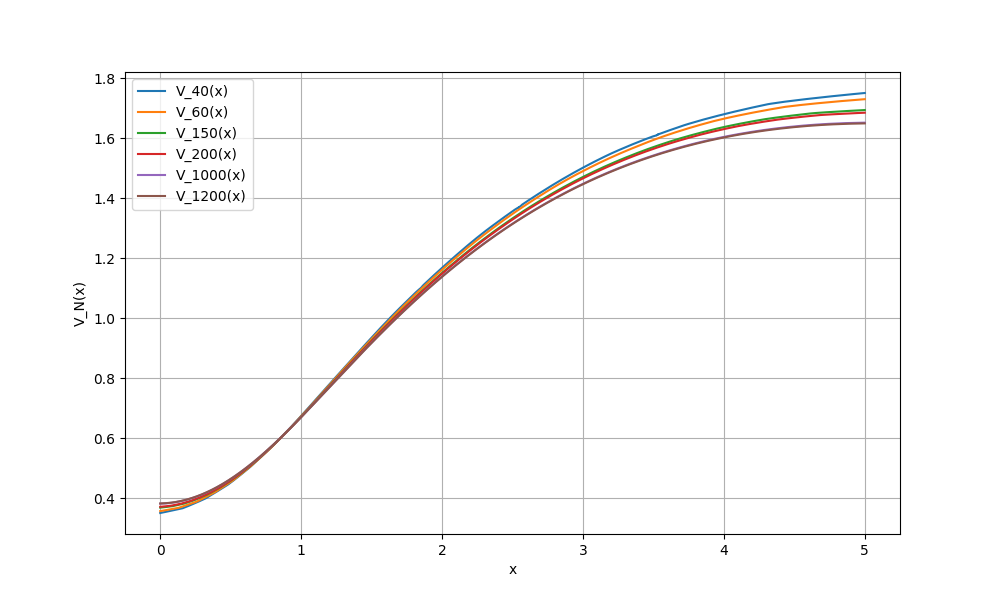} 
		\caption{Numerical experiment}
		\label{fig: num_ex}
	\end{figure}
	\begin{numexperiment}
		We consider the unit time horizon \(T = 1\), running payoff \(g \equiv 0\) and a terminal ``cut-off type call payoff'' \(\ell (x) := \min \{ \max \{ x - 0.5, 0\}, 20\}\). For the market model, we take 
		\[
	 a_* (x) := \max \{ \min \{ x, 30 \}, 1 \}, \quad a^* (x) := \big(\max \{ \min \{ x, 30 \}, 1 \} \big)^2.
		\] 
		Notice that 
		\begin{align*}
			[a_* (x), a^* (x)] = \Big\{  \max \{ \min \{ x, 30 \}, 1 \}^\lambda \colon \lambda \in [1, 2] \Big\},
		\end{align*}
		which means we consider a type of cut-off CEV model with uncertainty in the power parameter (in the classical CEV model power parameters larger than \(1\) occur for example in commodity markets; for instance, \cite[p.~83]{GS07} find a parameter around \(1.4\) for copper and gold). 
		Figure~\ref{fig: num_ex} shows numerical approximations of the iterations \(V_N (x) := S^N_{1/N} (x)\) for \(N = 40, 60, 150, 200, 1000, 1200\) computed on the grid \((0, 5, 5000 \text{ increments})\). It illustrates that the iteration convergences. Indeed, the last two curves have almost no visible difference.
	\end{numexperiment}
\begin{remark}
	The formula \eqref{eq: approx formula} resembles numerical schemes for Hamilton--Jacobi--Bellman PDEs as studied for example in \cite{Camilli_Falcone,Menaldi_89}. To be more specific, it corresponds to the time discretization scheme from \cite[Remark~2.2]{Camilli_Falcone}. Convergence rates of such schemes were studied in~\cite{Barles_Jakobsen}.
	
	Further, it is also interesting to make the connection to Chernoff approximations for convex semigroups that were considered in \cite{BKN23}.
	In principle, the Chernoff approximation provides the formula \eqref{eq: approx formula} for a terminal payoff \(f = \ell (X_T)\). A recent study of convergence rates can be found in \cite{blessing_etal_23}.
\end{remark}

\subsection{Convergence of Robust Superhedging Prices} \label{sec: app2 conv SHP}
A natural question in the context of market model convergence is whether it preserves some structure. 
For single prior models such properties were discussed in \cite{DP92} or Section~6.4 from \cite{BK}. To give an example, it is known that binomial approximations lead to the convergence of price processes, cf. \cite[Proposition~6.4.1]{BK}. In this section, we establish a similar result for our robust framework.

We consider a binomial approximation in the spirit of Example~\ref{ex: BinM}. Let \(b^*, b_* \colon \bR \to \bR\) and \(a^*, a_* \colon \bR \to (0, \infty)\) be Lipschitz continuous functions such that, for a constant \(C > 0\), 
\[
- C \leq b_* \leq b^* \leq C, \quad \frac{1}{C} \leq a_* \leq a^* \leq C.
\]
We consider the action space
\begin{align*}
	\Lambda := [0, 1] \times [0, 1] \times [0, 1] \times [0, 1],
\end{align*}
and, for \((\lambda = (\lambda_1, \dots, \lambda_4), x) \in \Lambda \times \bR\) and \(h > 0\), we define 
\begin{align*}
	\Phi_h (\lambda, x, dy) :=p (\lambda, x) \, &\delta_{x + u_h (\lambda, x)} (dy) + (1 - p (\lambda, x))\, \delta_{x + d_h (\lambda, x)} (dy), 
\end{align*}
where 
\begin{align*}
	p (\lambda, x) &:= \frac{\sqrt{a_* (x)  + \lambda_4 (a^* (x) - a_* (x))}}{\sqrt{a_* (x)  + \lambda_4 (a^* (x) - a_* (x))} + \sqrt{a_* (x)  + \lambda_3 (a^* (x) - a_* (x))}}, \\
	u_h (\lambda, x) &:=  h\, ( b_* (x) + \lambda_1 (b^* (x) - b_* (x))) + \sqrt{h} \, \sqrt{a_* (x)  + \lambda_3 (a^* (x) - a_* (x))}, \phantom \int
	\\
	d_h (\lambda, x) &:=  h\, ( b_* (x) + \lambda_2 (b^* (x) - b_* (x))) -  \sqrt{h} \, \sqrt{a_* (x) + \lambda_4 (a^* (x) - a_* (x))}. \phantom \int
\end{align*}
For a fixed initial value \(x_0 \in \bR\), let \(\mathscr{G}^h\) be the set of all probability measures \(P\) on \((\Sigma^1, \cG^1)\) such that 
\begin{align*} 
	P \Big( \underset{k = 0}{\overset{\infty} \Cross}  d \omega_k \Big) = \delta_{x_0} (d \omega_0) \, \prod_{k = 1}^\infty P^k (\omega_0, \dots, \omega_{k - 1}, d \omega_k), 
\end{align*}
where \(P^n\) is a measurable selector of \((x_1, \dots, x_{n - 1}) \mapsto \{\Phi_h (\lambda, x_{n - 1}, dy) \colon \lambda \in \Lambda\}\). 
We now define the corresponding discrete robust superhedging price. 

For a set \(\mathscr{P}\) of probability measures, we say that a property holds {\em \(\mathscr{P}\)-quasi surely}, written \(\mathscr{P}\)-q.s., if it holds outside a \(\mathscr{P}\)-polar set, i.e., outside a set \(A\) that satisfies \(\sup_{P \in \mathscr{P}} P (A) = 0\).
Let \(\mathcal{H}_{\on{d}}\) be the set of all predictable processes \(H = (H_k)_{k = 1}^\infty\) on the setup \((\Sigma^1, \cG^1, (\cG^1_n)_{n = 0}^\infty)\).

As in the previous section, to ease our presentation, we restrict our attention to the payoff function \(f\) from \eqref{eq: payoff f}. Let \((f^h)_{h > 0}\) be the approximation payoff functions as in \eqref{eq: payoff approx} defined on the discrete time path space \(\Sigma^1\), i.e., 
\begin{align*}
	f^h := h\sum_{i = 0}^{\lfloor T / h \rfloor - 1} g(Y_i) + \ell (Y_{\lfloor T / h \rfloor}).
\end{align*}
The discrete robust superhedging price is given by 
\begin{align*}
	\pi_{\on{d}}^h (f^h) := \inf \Big\{ x \in \bR \colon \exists \, H &\in \mathcal{H}_{\on{d}} \text{ with } 
	x + \sum_{k = 1}^{\lfloor T / h \rfloor - 1} H_k (Y_k - Y_{k -1}) \geq f^h \ \, \mathscr{G}^h\text{-q.s.} \Big\}.
\end{align*}
Moreover, let \(\pi (f)\) be the continuous robust superhedging price from \eqref{eq: superhedging duality}.
We now have the following convergence result.
\begin{theorem} \label{theo: conv superhed}
As \(h \searrow 0\),
\[
\pi^h_{\on{d}} (f^h) \to \pi (f). 
\]
\end{theorem}
\begin{proof}
	The tactic is to use robust superhedging dualities and to investigate the convergence of the corresponding worst-case expectations. 
	
	Recall from the proof of Proposition~\ref{prop: computation RSP} that 
	\[
	\pi (f) = \sup_{Q \in \mathscr{M}} E^Q \big[ f \big], 
	\]
	where \(\mathscr{M}\) is given by \eqref{eq: concrete}. In the following we establish a similar representation for the discrete robust superhedging price.
	We define the (now discrete time) set of equivalent martingale measures by 
	\begin{align*}
		\mathscr{M}^h_{\on{d}} := \Big\{ Q \in \mathcal{P} (\Sigma^1) \colon \exists \, P \in \mathscr{G}^h \text{ with } P \sim_\textup{loc} Q \text{ and \(Y\) is a \(Q\)-martingale} \Big\}.
	\end{align*}
	The following is a version of the robust superhedging theorem from \cite{BN15}. Notice that we cannot apply the result from \cite{BN15} directly, as our domain of uncertainty is not convex. Still, the arguments from \cite{BN15} can be adapted to our setting. A proof of the following lemma is given after the proof of Theorem~\ref{theo: conv superhed} is complete.
\begin{lemma} \label{lem: superhedging duality discrete}
	For \(h > 0\) small enough, 
	\begin{align}\label{eq: superhedging duality discrete}
		\pi^h_{\on{d}} (f^h) = \sup_{Q \in \mathscr{M}^h_{\on{d}}} E^Q \big[ f^h \big].	
		\end{align}
		\end{lemma} 
	Our next goal is to get a tractable description of the right hand side in \eqref{eq: superhedging duality discrete}.
	Define the set \(\ccC^h (x_0)\) as in Section~\ref{sec: MC Di} with
	\begin{align*}
	\Pi_h (\lambda, x, dy) &:= \frac{d_h (\lambda, x_2)}{ d_h (\lambda, x_2) - u_h (\lambda, x_2) } \, \delta_{x + (-v (\lambda, x_2) h /d_h (\lambda, x_2) , u_h (\lambda, x_2))} (dy) 
	\\&\quad\qquad+ \frac{u_h (\lambda, x_2)}{u_h (\lambda, x_2) - d_h (\lambda, x_2)}\, \delta_{x + (-v (\lambda, x_2) h / u_h (\lambda, x_2), d_h (\lambda, x_2))} (dy),
	\end{align*} 
	where \(h > 0\), small enough, \((\lambda = (\lambda_1, \dots, \lambda_4), x = (x_1, x_2)) \in \Lambda \times \bR^2\) and 
	\[
	v (\lambda, x_2) := \Big[ (a_* (x_2)  + \lambda_3 (a^* (x_2) - a_* (x_2))) (a_* (x_2)  + \lambda_4 (a^* (x_2) - a_* (x_2))) \Big]^{\frac{1}{4}}.
	\]
	Here, we emphasize that the setting is now two-dimensional.
			
			Recall that \(\on{pr}_{\Omega^1} \colon \Omega^2 \to \Omega^1\) denotes the projection to the second coordinate.
	It is not hard to see that 
	\begin{align} \label{eq: inclusion}
	\mathscr{M}^h_{\on{d}} \subset \Big\{ Q \circ \on{pr}_{\Omega^1}^{-1} \colon Q \in \ccC^h (x_0) \Big\},
	\end{align}
	and consequently, 
	\begin{align*} 
	\sup_{Q \in \mathscr{M}^h_{\on{d}}} E^Q \big[ f^h \big] \leq \sup_{Q \in \ccC^h (x_0)} E^Q \big[ f^h (\on{pr}_{\Omega^1}) \big].
	\end{align*} 
	The converse inclusion in \eqref{eq: inclusion} seems not to hold in general. 
	Roughly speaking, the difference between the sets in \eqref{eq: inclusion} are the admissible controls, which are allowed to have more flexibility for the right hand set. On the level of the worst-case expectations, this flexibility is not crucial. 
	To wit, by \cite[Theorem~2.2.3]{bauerle_rieder} (or see \cite[Proposition~8.1]{bershre}), for every \(Q \in \ccC^h (x_0)\) there exists a measure \(Q^* \in \mathscr{M}^h_{\on{d}}\) such that 
	\begin{align*}
		 E^{Q^*} \big[ f^h\big] =E^Q \big[ f^h (\on{pr}_{\Omega^1}) \big].
	\end{align*} In summary, 
	\begin{align*}
			\sup_{Q \in \mathscr{M}^h_{\on{d}}} E^Q \big[ f^h \big] = \sup_{Q \in \ccC^h (x_0)} E^Q \big[ f^h (\on{pr}_{\Omega^1}) \big].
	\end{align*}
	Notice that 
	\begin{align*}
		\int (x_i - y_i) \, \Pi_h (\lambda, x, dy) = 0, \quad i = 1, 2, 
	\end{align*}
	and that 
	\begin{align*}
	\frac{1}{h}	\int (x_1 - y_1)^2\, \Pi_h (\lambda, x, dy) &= \frac{- [v (\lambda, x_2)]^2}{ w_h (\lambda, x_2)} \xrightarrow{h \, \searrow \, 0} 1,
	\\ 
	\frac{1}{h}		\int (x_1 - y_1)(x_2 - y_2)\, \Pi_h (\lambda, x, dy) &= v (\lambda, x_2), \\
	\frac{1}{h}		\int (x_2 - y_2)^2\,  \Pi_h (\lambda, x, dy) &= - \frac{u_h (\lambda, x) \, d_h (\lambda, x)}{h}  \xrightarrow{h \, \searrow \, 0} \big[ v (\lambda, x_2) \big]^2,
	\end{align*}
	where
	\begin{align*}
	w_h (\lambda, x_2) :=  \big[\sqrt{h}\, ( &b_* (x) + \lambda_1 (b^* (x) - b_* (x))) + \sqrt{a_* (x)  + \lambda_3 (a^* (x) - a_* (x))}\, \big] \\&\cdot \big[  \sqrt{h} \, ( b_* (x) + \lambda_2 (b^* (x) - b_* (x))) - \sqrt{a_* (x)  + \lambda_4 (a^* (x) - a_* (x))}\, \big].
	\end{align*}
	Consequently, straightforward considerations show that Theorem~\ref{theo: main1} implies
	\begin{align*}
	\sup_{Q \in \ccC^h (x_0)} E^Q \big[ f^h (\on{pr}_{\Omega^1}) \big] \to \sup_{Q \in \mathscr{M}} E^Q \big[ f \big], \quad h \searrow 0.
	\end{align*} 
	This proves our claim.
\end{proof}

\begin{proof}[Proof of Lemma~\ref{lem: superhedging duality discrete}]
		For a set \(G\), we denote its convex hull by \(\on{co} \, (G)\). 
		Moreover, for a probability measure \(Q\) and a set \(\mathscr{K}\) of probability measure on the same underlying measurable space, we write \(Q \lll \mathscr{K}\) if there exists a \(P \in \mathscr{K}\) such that \(Q \ll P\).
		
		Let us start with two observations. 
	
	\smallskip
	{\em Observation 1:} 
	For any set \(\mathscr{K}\) of probability measures and any bounded universally measurable function~\(g\), it is easy to see that 
	\begin{align*}
		\sup_{Q \in \mathscr{K}} E^Q \big[ g \big] = \sup_{Q \in \on{co} (\mathscr{K})} E^Q \big[ g \big].
	\end{align*}
	As a consequence, \(\mathscr{K}\) and \(\on{co} (\mathscr{K})\) have the same polar sets. 
	
	\smallskip
	{\em Observation 2:} For \(d_* \leq d^* < 0 < u_* \leq u^*\), define 
		\begin{align*}
		\mathscr{P} := \Big\{\, p\, \delta_{u} + (1 - p) \, \delta_{d}  \colon 0 < p = p(u, d) < 1, \
		 u \in \big [ u_*, u^* \big], \,
		d \in \big[ d_*, d^* \big]  \Big\},
	\end{align*}
	and 
	\begin{align*}
		\mathscr{Q} &:=	\Big\{ Q \in \mathcal{P} (\bR) \colon Q \lll \mathscr{P} \text{ and } \int_{- \infty}^\infty x \, Q (dx) = 0 \Big\}, \\
		\mathscr{Q}^{\on{co}} &:=	\Big\{ Q \in \mathcal{P} (\bR) \colon Q \lll \on{co} \, (\mathscr{P}) \text{ and } \int_{- \infty}^\infty x\, Q (dx) = 0 \Big\}.
	\end{align*}
	The set \(\mathscr{P}\) is not defined in an unambiguous manner, as we do not describe how the weight \(p(u, d)\) depends on \(u\) and \(d\). The crucial properties we need are that the points \(u\) and \(d\) are taken from a product set and that both Dirac measures are ``active'' in the sense that they contribute to the support.
	We claim that
	\begin{align} \label{eq: identity convex hull}
	\on{co} \, (\mathscr{Q}) = \mathscr{Q}^{\on{co}}.
	\end{align} 
	The inclusion \(\subset\) is easy to see. To prove the converse, take \(Q \in \mathscr{Q}^{\on{co}}\), i.e., 
	\[
	Q = \sum_{i = 1}^n \Big( \lambda_i \delta_{u_i} + \mu_i \delta_{d_i}\Big), \quad \int_{- \infty}^\infty x \, Q(dx) = 0, 
	\]
	with
	\[
	\lambda_i, \mu_i > 0, \ \sum_{i = 1}^n (\lambda_i + \mu_i) = 1, \ u_1 \leq u_2 \leq \dots  \leq u_n < 0 < d_n \leq d_{n - 1} \leq \dots \leq d_1.
	\]
	We now use the following algorithm: Take \((u_1, d_1)\) and \((\lambda_1, \mu_1)\). 
\begin{enumerate}
\item[-]	If \(\lambda_1 u_1 + \mu_1 d_1 = 0\), set \[Q_1 := \frac{\lambda_1 \delta_{u_1} + \mu_1 \delta_{d_1}}{\lambda_1 + \mu_1}.\]
Then, proceed with \((u_2, d_2)\) and \((\lambda_2, \mu_2)\).
\item[-]
	If  \(\lambda_1 u_1 + \mu_1 d_1 < 0\), take \(\tilde{\mu}_1 := - \lambda_1 u_1 / d_1 \in (0, \mu_1)\) such that \(\lambda_1 u_1 + \tilde{\mu}_1 d_1 = 0\) and set 
	\[Q_1 := \frac{\lambda_1 \delta_{u_1} + \tilde{\mu}_1 \delta_{d_1}}{\lambda_1 + \tilde{\mu}_1}.\]
	Then, proceed with \((u_2, d_1)\) and \((\lambda_1, \mu_1 - \tilde{\mu}_1)\).
\item[-]
	If  \(\lambda_1 u_1 + \mu_1 d_1 > 0\), take \(\tilde{\lambda}_1 := - \mu_1 d_1 / u_1 \in (0, \lambda_1)\) such that \(\tilde{\lambda}_1 u_1 + \mu_1 d_1 = 0\) and set 
	\[Q_1 := \frac{\tilde{\lambda}_1 \delta_{u_1} + \mu_1 \delta_{d_1}}{\tilde{\lambda}_1 + \mu_1}.
	\]
	Then, proceed with \((u_1, d_2)\) and \((\lambda_1 - \tilde{\lambda}_1, \mu_2)\).
\end{enumerate}
Because \(\int_{- \infty}^\infty x\, Q (dx) = 0\), this algorithm produces a representation
\[
Q = \sum_{i = 1}^m \gamma_i \, Q_i, \quad m \geq n, \ \gamma_i \geq 0, \ \sum_{i = 1}^m \gamma_i = 1, \ Q_i \in \mathscr{Q},
\] 
and consequently, we get that \(Q \in \on{co} \, (\mathscr{Q})\). This finishes the proof of \eqref{eq: identity convex hull}.

\smallskip 
By virtue of \cite[Theorem~4.5]{BN15}, the ``no arbitrage'' condition from \cite[Definition~1.1]{BN15} holds for the uncertainty set build with \(\on{co} \, (\{\Phi_h (\lambda, x, dy) \colon \lambda \in \Lambda\})\) instead of \(\{\Phi_h (\lambda, x, dy) \colon\) \(\lambda \in \Lambda\}\). Consequently, the prerequisites of \cite[Lemma~4.10]{BN15} are satisfied (for the ``convexified setting''). Using both of the above observations, we can apply \cite[Lemma~4.10]{BN15} also to our non-convex setting. Then, arguing as in the proof of \cite[Lemma~4.13]{BN15}, the robust superhedging duality follows also for our discrete market model. We omit the details for brevity.
\end{proof} 
 
\section{Martingale and Control Representations} \label{sec: MP relation}

In the first part of this section we explain equivalent characterizations of the uncertainty sets \(\ccP^h\) and \(\cD\) that adapt ideas from \cite{C23b}. In the second part we recall the relaxed control representation of the set \(\cD\) that was established in \cite{CN23a}.
These observations are used in the following proof-sections. We stress that our Standing Assumptions~\ref{SA: 1} -- \ref{SA: 3} are in force in this section.

\subsection{Martingale Characterizations} \label{sec: characterization 1}

The pregenerator of a Markov chain with transition kernel \(\Pi_h (\lambda, \, \cdot\,, \, \cdot\,)\) is given by 
\[
A_h (f) (\lambda, x) := \int (f (y) - f (x))\, \Pi_h (\lambda, x, dy),
\]
for \(f \in C_b (\bR^{d}; \bR)\) and \((\lambda, x) \in \Lambda \times \bR^{d}\). We restrict ourselves to a countable class of test functions that is given by
\begin{align*} 
	\mathcal{T}^d := \Big\{ \sin ( \langle u, \, \cdot\,\rangle) \colon u \in \mathbb{Q}^{d} \Big\} \cup  \Big\{ \cos (\langle u, \, \cdot \,\rangle) \colon u \in \mathbb{Q}^{d} \Big\}\subset C^2_b (\bR^{d}; \bR).
\end{align*}
Let \((\t_1, \t_2, \dots )\) be an enumeration of \(\mathcal{T}^d\) and define 
\begin{align*}
	\Xi^h (x) := \Big\{ \big( A_h (\t_k) (\lambda, x) \big)_{k = 1}^\infty \colon \lambda \in \Lambda \Big\} \subset \bR^\N
\end{align*}
for \(h > 0\) and \(x \in \bR^{d}\). 
For a sufficiently integrable process \(Z = (Z_n)_{n = 0}^\infty\), we set
\[
\A^P_0 (Z) := 0, \quad \A^P_n (Z) := \sum_{i = 1}^n E^P \big[ Z_i - Z_{i-1} \mid \cG_{i - 1} \big], \ \ n \geq 1.
\]
Evidently, the process \(Z - \A^P (Z)\) is a local martingale, meaning that \(\A^P (Z)\) is the predictable compensator of \(Z\).

Now, for \(x \in \bR^m\), we define 
\begin{align*}
	\cR^h (x) := \Big\{ P \in \mathcal{P} (\Sigma^d) \colon & P (Y_0 = (0_r, x)) = 1, P\text{-a.s. for all } k \in \N \\& \big( \A^P_{k} (\t_n (Y)) - \A^P_{k-1}(\t_n(Y)) \big)_{n = 1}^\infty \in \Xi^h (Y_{k - 1}) \, \Big\}. 
\end{align*}
The set \(\cR^h (x)\) admits a feedback control representation, which explains its connection to \(\cC^h (x)\).
 
	Using Filippov's implicit function theorem (\cite[Theorem~18.17]{charalambos2013infinite}), it can be shown that for every \(P \in \cR^h (x)\) there exists a \(\Lambda\)-valued adapted process \(\f = (\f_n)_{n = 0}^\infty\) such that, for all \(k \in \mathbb{N}\), the process
	\[
	\Big( \t_k (Y_n) - \sum_{i = 0}^{n - 1} A_h (\t_k) (\f_i, Y_i) \Big)_{n = 0}^\infty 
	\]
	is a \(P\)-martingale, cf. the proof of \cite[Proposition~2.4]{C23b} for some details. Consequently, by virtue of the definition of \(\mathcal{T}^d\), it follows that \(P\)-a.s.
	\begin{align*}
		E^P \Big[ e^{\mathsf{i} \langle u, Y_k \rangle } \mid \cG_{k - 1} \Big] = \int \, e^{\mathsf{i} \langle u, y\rangle}\, \Pi_h (\f_{k- 1}, Y_{k - 1}, dy)
	\end{align*}
	for all \(k \in \mathbb{N}\) and \(u \in \mathbb{Q}^{d}\). Here, \(\mathsf{i}\) denotes the imaginary number.
	A density argument shows that this identity also holds for \(u \in \bR^{d}\) instead of \(u \in \mathbb{Q}^{d}\). In other words, \(P\)-a.s.
	\begin{align} \label{eq: formula cond prob}
		P ( Y_k \in dy \mid \cG_{k - 1} ) = \Pi_h (\f_{k - 1}, Y_{k - 1}, dy), \quad k \in \mathbb{N}.
	\end{align}
As a corollary of this discussion, we immediately get the following:
\begin{lemma} \label{coro: connection}
	\(\ccP^h (x) = \cR^h (x)\) for all \(h > 0\) and \(x \in \bR^m\).
\end{lemma}

Next, we also explain a similar representation of the set \(\cD (x)\).
For a given probability measure \(P\) on \((\Omega^{m}, \cF^{m})\), let \(\Sc (P)\) be the set of all real-valued continuous (in time and path) semimartingales on the stochastic basis \((\Omega^{m}, \cF^{m}, (\cF^{m}_t)_{t \geq 0}, P)\). For \(Z \in \Sc (P)\), let \(\A^P (Z)\) be the unique (up to a \(P\)-null set) predictable compensator of \(Z\), that is presumed to start in the origin by convention.
Let \(\Sac (P)\) be the set of all \(Z \in \Sc (P)\) such that \(\A^P (Z)\) is \(P\)-a.s. absolutely continuous w.r.t. the Lebesgue measure.
For \(f \in C^2 (\bR^{m}; \bR)\), define
\[
L (f) (\lambda, x) := \langle \nabla f (x), b (\lambda, x) \rangle + \tfrac{1}{2} \on{tr} \big[ \nabla^2 f (x) \, \sigma \sigma^* (\lambda, x) \big], \quad (\lambda, x) \in \Lambda \times \bR^m,
\] 
and, for an enumeration \(\s_1, \s_2, \dots\) of the set \(\mathcal{T}^m\), we define
\begin{align*}
	\Psi   (x) := \Big\{ \big( L (\s_k) (\lambda, x) \big)_{k = 1}^\infty \colon \lambda \in \Lambda \Big\} \subset \mathbb{R}^\mathbb{N}, \quad x \in \bR^{m}.
\end{align*}
We have the following:
\begin{lemma}  \label{lem: MP chara d set}
	For every \(x \in \bR^m\), 
	\begin{equation*} 
		\begin{split}
			\cD (x) = \Big\{ P \in \mathcal{P} (\Omega^{m}) \colon & P (X_0 = x) = 1, \, \{ \s_n (X) \colon n \in \N \} \subset \Sac (P), \\ 
			& (\llambda \otimes P)\text{-a.e. } \big( \A^P (\s_k) / d \llambda \big)_{k = 1}^\infty \in \Psi (X) \Big\}.
		\end{split}
	\end{equation*}
\end{lemma} 
\begin{proof}
	The inclusion ``\(\subset\)'' is clear. The converse follows as in Section~3.1 from \cite{C23b}. We skip the details for brevity.
\end{proof}

\subsection{Relaxed Control Representation} \label{sec: characterization 2}
The purpose of this section is to recall the relation of the set \(\cD\) to relaxed control rules as used in \cite{nicole1987compactification,EKNJ88,ElKa15}.
	Let \(\m\) to be the set of all Radon measures on \(\bR_+ \times \Lambda\) whose projections to \(\bR_+\) coincide with the Lebesgue measure. 
We endow \(\m\) with the vague topology, which turns it into a compact metrizable space (\cite[Theorem~2.2]{EKNJ88}). The coordinate map on \(\m\) is denoted by \(M\).
Define the \(\sigma\)-field \[\M := \sigma (M_t (\phi); t \in \bR_+, \phi \in C_{c} (\bR_+ \times \Lambda; \bR)),\] where
\[
M_t (\phi) := \int_0^t \int \phi (s, \lambda) M(ds, d\lambda).
\]
Relaxed control rules will be defined as probability measures on the product space \((\Omega^d \times \m, \cF^d \otimes \M)\). Adapting our standard notation, we denote the coordinate process on this space by \((X, M)\). 
For \(f \in C^2_b (\bR^d; \bR)\), define
\begin{align} \label{eq: C martingale relax}
C (f) :=f  (X) - \int_{0}^{\cdot} \int \bar{L} (f) (\lambda, X_s) M (ds , d \lambda),
\end{align} 
with 
\begin{align} \label{eq: bar L}
\bar{L} (f) (\lambda, x) := \langle \nabla f (x), \ob (\lambda, x) \, \rangle + \tfrac{1}{2} \on{tr} \big[ \nabla^2 f (x) \, \oa (\lambda, x) \, \big], \quad (\lambda, x) \in \Lambda \times \bR^d.
\end{align} 
A \emph{relaxed control rule} with initial value \(x \in \bR^m\) is a probability measure \(P\) on the product space \((\Omega^d \times \m, \mathcal{F}^d \otimes \M)\) such that \(P (X_0 = (0_r, x) ) = 1\) and such that the processes \(C (f)\), with \(f \in C^2_b (\bR^d; \bR)\), are \(P\)-martingales for the filtration
\[
\cH_s := \sigma (X_r, M_r (\phi); r \leq s, \phi \in C_c (\bR_+ \times \Lambda; \bR)), \quad s \in \bR_+.
\]
Finally, for \(x \in \bR^m\), we define 
\begin{align*}
	\cK (x)&:= \Big\{ \text{all relaxed control rules with initial value } x \Big \}.
\end{align*}
With little abuse of notation,\footnote{We already defined \(\on{pr}_{\Omega^1}\) as the projection \(\Omega^2 \to \Omega^1\) to the second coordinate.} we define the projection \(\on{pr}_{\Omega^m} \colon \Omega^d \times \m \to \Omega^m\) by 
\[
\on{pr}_{\Omega^m} (\omega, m) := (\omega^{r + 1}, \dots, \omega^d)
\] 
for \((\omega, m) \in \Omega^d  \times \m\) with \(\omega = (\omega^1, \dots, \omega^d)\).

Recalling that Standing Assumption~\ref{SA: 3} is in force, the following representation for the set \(\cD (x)\) is a consequence of \cite[Theorem~5.1]{CN23a}.
\begin{lemma} \label{lem: relaxed controls}
	For all \(x \in \bR^m\), it holds that
	\[
	\cD (x) = \Big\{ P \circ \on{pr}_{\Omega^m}^{-1} \colon P \in \cK (x) \Big\}.
	\]
\end{lemma}

We are in the position to prove our main theorems.

\section{Proof of Theorem~\ref{theo: main1}~(ii)} \label{sec: pf}

The Hausdorff metric convergence can be proved by a similar strategy as continuity of set-valued maps with compact values. Namely, we establish properties that are related to the sequential characterizations of upper and lower hemicontinuity of set-valued maps. Lastly, we will combine these observations to establish the convergence from Theorem~\ref{theo: main1}~(ii). 

\smallskip
Throughout this section, we assume that the Standing Assumptions~\ref{SA: 1} -- \ref{SA: 3} are in force.

\subsection{The ``Upper'' Part} \label{sec: upper}
The main observation in this section is the following.
\begin{proposition} \label{prop: upper hemi}
	Let \((h_n)_{n = 1}^\infty \subset (0, \infty)\) be a sequence such that \(h_n \searrow 0\), and let \((x_n)_{n = 0}^\infty \subset \bR^m\) be a sequence such that \(x_n \to x_0\). Further, take a sequence \((Q^n)_{n = 1}^\infty\) such that \(Q^n \in \cC^{h_n} (x_n)\). 
	Then, the following hold:
	\begin{enumerate}
	\item[\textup{(i)}] 
	The family \(\{Q^n \colon n \in \N\}\) is relatively compact in \(\cP(\Omega^{m})\). 
	\item[\textup{(ii)}]
	Any accumulation point of \(\{ Q^n\colon n \in \mathbb{N}\}\) is contained in \(\cD (x_0)\).
\end{enumerate}
\end{proposition}

We will first prove part (i) of Proposition~\ref{prop: upper hemi} by means of the Stroock--Varadhan (\cite{SV}) martingale criterion for tightness. For part (ii) we present two different proofs that we think are both of independent interest. The first proof uses arguments from set-valued analysis, while the second proof relies on the relaxed control framework. 

We start with some technical observations. 
Let \(\delta > 0\) be as in Standing~Assumption~\ref{SA: 2} and let \(\theta \in (0, \delta)\) be small enough such that 
\begin{align} \label{eq: theta bounded}
\sup \Big\{ \tfrac{1}{h} \, \Pi_h (\lambda, x, \{y \colon \|y - x\| > 1\}) \colon h < \theta, \, (\lambda, x) \in \Lambda \times \bR^{d} \Big\} < \infty.
\end{align}
Such a constant \(\theta\) exists due to the final part of Standing Assumption~\ref{SA: 2}. 
	Without loss of generality, we always assume that \(h_1 < \theta\). 
	
\begin{lemma} \label{lem: uni bdd comp dis}
	For \(g \in C^2_b (\bR^{d}; \bR)\), there exists a constant \(C = C (g) > 0\) such that 
	\begin{align*}
		\sup \Big\{ \tfrac{1}{h} | A_h (g) (\lambda, x) | \colon h < \theta, \ (\lambda, x) \in \Lambda \times \bR^{d} \Big\}  \leq C. 
	\end{align*}
	Furthermore, we can take the constant such that \(C (g) = C (g + \ell)\) for all \(\ell \in \bR\).
\end{lemma}
\begin{proof}
	Using Taylor's theorem, we get that 
	\begin{align*}
		| A_h (g) (\lambda, x) | &\leq \int_{\|x - y\| \leq 1} \big| g (y) - g (x) - \langle \nabla g (x), y - x\rangle \big| \Pi_h (\lambda, x, dy) 
		\\&\hspace{2cm}+ \Big| \Big \langle \nabla g (x), \int_{\|x - y\| \leq 1} (y - x) \Pi_h (\lambda, x, dy) \Big \rangle \Big|
		\\&\hspace{2cm}+ \int_{\|x - y\| > 1} | g (y) - g (x)| \Pi_h (\lambda, x, dy)
		\\&\leq_c \int_{\|x - y\| \leq 1} \|y - x\|^2\, \Pi_h (\lambda, x, dy) + \Big\| \int_{\|x - y\| \leq 1} (y - x) \Pi_h (\lambda, x, dy) \Big\| 
		\\&\hspace{2cm}+ \Pi_h (\lambda, x, \{y \colon \|x - y\| > 1\}) \phantom \int
		\\&\leq_c h \big( \|a_h (\lambda, x)\| + \|b_h (\lambda, x)\| \big) + \Pi_h (\lambda, x, \{y \colon \|x - y\| > 1\}), \phantom \int
	\end{align*}
	where \(\leq_c\) denotes an inequality up to a multiplicative positive constant (that is independent of \(\lambda\) and \(x\)). 
	Now, the existence of the constant \(C\) follows from Standing Assumption~\ref{SA: 2} and the definition of \(\theta\).
	Further, that the final constant is invariant under translation is obvious from the above estimate. 
	The proof is complete.
\end{proof}

\begin{proof}[Proof of Proposition \ref{prop: upper hemi} (i)]
	First, we ``extend'' our setting and consider all \(d\) coordinates from the sets~\(\cR^h (x)\). Define the linear interpolation map
\begin{align} \label{eq: interpol complete}
	\Y_t (h) := \Big( \Big\lfloor \, \frac{t}{h}\, \Big \rfloor + 1 - \frac{t}{h} \Big)  \, Y_{\lfloor t/h\rfloor} + \Big( \frac{t}{h} - \Big\lfloor \, \frac{t}{h}\, \Big \rfloor\Big) \,  Y_{\lfloor t/h \rfloor + 1}, \quad t \geq 0,
\end{align}
which maps \(\Sigma^d\) into \(\Omega^d\). By definition of \(\cC^h (x)\) and Lemma~\ref{coro: connection}, there exists a sequence \((R^n)_{n = 1}^\infty\) with \(R^n \in \cR^{h_n} (x_n)\) and \(Q^n = R^n \circ \X (h_n)^{-1}\). Set \(P^n := R^n \circ \Y (h_n)^{-1} \in \cP (\Omega^d)\). In the following, we show that \(\{P^n \colon n \in \N\}\) is relatively compact in \(\cP (\Omega^d)\). As \(Q^n = P^n \circ \on{pr}_{\Omega_m}^{-1}\), where \(\on{pr}_{\Omega_m} \colon \Omega^d \to \Omega^m\) denotes the projection \(\on{pr}_{\Omega^m} (\omega^1, \dots, \omega^d) := (\omega^{r + 1}, \dots, \omega^d)\), which is clearly continuous, this implies relative compactness of \(\{Q^n \colon n \in \N\}\). 

\smallskip 
To show relative compactness of \(\{P^n \colon n \in \N\}\), we use \cite[Theorem~1.4.11]{SV}. 
	By virtue of the formula \eqref{eq: formula cond prob}, we get that
	\begin{align*}
	\sum_{j = 0}^{\lfloor T / h_n \rfloor} P^n ( \| &X_{(j + 1) h_n} - X_{j h_n} \| \geq \varepsilon )
	\\&= \sum_{j = 0}^{\lfloor T / h_n \rfloor} E^{P^n} \big[ P^n ( \| X_{(j + 1) h_n} - X_{j h_n} \| \geq \varepsilon \mid \sigma (X_{ih_n}, i = 1, \dots, j)) \big]
	\\& =\sum_{j = 0}^{\lfloor T / h_n \rfloor} E^{P^n} \big[ \Pi_{h_n} (\f^n_{j h_n}, X_{j h_n}, \{ y \colon \|y - X_{jh_n}\| \geq \varepsilon \} ) \big] 
	\\&\leq \frac{(T + 1)}{h_n} \sup \Big\{ \Pi_{h_n} (\lambda, x, \{y \colon \|y - x\| \geq \varepsilon \} ) \colon (\lambda, x) \in \Lambda \times \bR^d\Big\} \phantom {\sum_{j = 0}^{\lfloor T / h_n \rfloor}}
	\\&= (T + 1) \Delta^\varepsilon_{h_n} \to 0, \phantom {\sum_{j = 0}}
	\end{align*}
	where the convergence follows from Standing Assumption~\ref{SA: 2}. Now, thanks to Lemma~\ref{lem: uni bdd comp dis} and the definition of the sets \(\cR^{h_n} (x_n)\), we can apply \cite[Theorem~1.4.11]{SV} and conclude that \(\{P^n \colon n \in \N\}\) is tight  or equivalently, relatively compact in \(\cP(\Omega^{d})\).
	\end{proof} 
\begin{proof}[First proof of Proposition~\ref{prop: upper hemi}~(ii)] 
	Without loss of generality, we may assume that \(Q^n \to Q^0\) in \(\cP (\Omega^{m})\). We need to show that \(Q^0 \in \cD (x_0)\). Notice that \(Q^0 (X_0 = x_0) = 1\) by the continuity of \(\omega \mapsto \omega (0)\).
	
	To establish the remaining properties of \(Q^0\), we pass again to the extended setting that was used in the proof of part (i). Let \(R^n \in \cR^{h_n} (x_n)\) be such that \(Q^n = R^n \circ \X (h_n)^{-1}\) and set \(P^n := R^n \circ \Y (h_n)^{-1}\). As \((P^n)_{n = 1}^\infty\) is tight by part (i), possibly passing to a convergent subsequence, we may assume that \(P^n \to P^0\) in \(\cP(\Omega^d)\). Evidently, we have \(Q^0 = P^0 \circ \on{pr}_{\Omega^m}^{-1}\), which shows that we may identify properties of \(Q^0\) via the measure \(P^0\).
	
	In the following we work on the path space \(\Omega^d\) whose coordinate process is denoted by~\(X\). Recall that \(\pi_m (x_1, \dots, x_d) = (x_{r+1}, \dots, x_{d})\) for \((x_1, \dots, x_d) \in \bR^d\).
	Define \(\phi^n_t := \lfloor t / h_n \rfloor + 1, {^n}\hspace{-0.05cm}Z := X_{\phi^n h_n}\) and 
	\begin{align*}
		\nk \A_t := \sum_{i = 1}^{\phi^n_t} E^{P^n} \big[ \s_k (\pi_m (X_{i h_n})) - \s_k (\pi_m(X_{(i - 1)h_n})) \mid \cF^{d}_{(i - 1)h_n} \big], \quad t \geq 0.
	\end{align*}
	Let \(\D\) be the Skorokhod space of \cadlag functions from \(\bR_+\) into \(\bR\). As usual, we endow \(\D\) with the Skorokhod \(J_1\) topology. 
	Evidently, each \(\nk \A\) has paths in \(\D\).
	
	\begin{lemma} \label{lem: L-tightness}
		For all \(k \in \N\), the family \(\{P^n \circ \nk \A^{-1} \colon n \in \mathbb{N}\}\) is relatively compact in~\(\cP(\D)\). Furthermore, any of its accumulation points is supported on the space of Lipschitz continuous functions \(\bR_+ \to \bR\). 
	\end{lemma}
	\begin{proof}
		By Lemma~\ref{lem: uni bdd comp dis}, for every \(T > 0\), we have \(P^n\)-a.s.
		\begin{equation*} \begin{split}
			\on{Var}_T \, (\nk \A) &= \sum_{i = 1}^{\phi^n_T} \big| E^{P^n} \big[ \s_k (\pi_m (X_{i h_n})) - \s_k (\pi_m(X_{(i - 1)h_n})) \mid \cF^{d}_{(i - 1)h_n} \big] \big| 
			\\&\leq (\phi^n_T - 1) \, h_n \, C \leq T\, C,  \phantom \int
		\end{split} \end{equation*}
		where \(C = C (\s_k)\) is the constant from Lemma~\ref{lem: uni bdd comp dis}, and \(\on{Var}_T \, (\, \cdot\,)\) denotes the variation on the time interval \([0, T]\). Thanks to this observation, we may conclude from \cite[Propositions~VI.3.35, VI.3.36]{JS} that \(\{ P^n \circ \nk \A^{-1} \colon n \in \N\}\) is relatively compact in \(\cP(\D)\) and that any of its accumulation points is supported on \(\Omega^{1} \equiv \Omega\). 
		
		\smallskip 
		It remains to show that any accumulation point is not only supported on \(\Omega\) but even on the space of Lipschitz continuous functions. For \(\omega \in \D\), define 
		\[
		\iii \omega \iii := \sup \Big\{ \frac{|\omega (t) - \omega (s)|}{|t - s|} \colon s \not = t \Big\}.
		\]
		We shortly explain that \(\omega \mapsto \iii \omega\iii\) is lower semicontinuous from \(\D\) into \([0, \infty]\) at every \(\omega \in \Omega\). Take \((\omega^n)_{n = 1}^\infty \subset \D\) and \(\omega^0 \in \Omega\) such that \(\omega^n \to \omega^0\) in the Skorokhod \(J_1\) topology. 
		As \(\omega^0 \in \Omega\), the convergence \(\omega^n \to \omega^0\) also holds in the local uniform topology (cf. \cite[Proposition~VI.1.17]{JS}).
		Hence, we obtain that 
		\begin{align*}
			\iii \omega^0 \iii &= \sup \Big\{ \liminf_{n \to \infty} \, \frac{|\omega^n (t) - \omega^n (s)|}{|t - s|} \colon s \not = t \Big\}
			\leq \liminf_{n \to \infty} \iii \omega^n\iii.
		\end{align*}
		Let \((K^n)_{n = 1}^\infty\) be a convergent subsequence of \((P^n \circ \nk \A^{-1})_{n = 1}^\infty\) with limit~\(K^0\). Because \(K^0 (\Omega) = 1\), we showed that \(\omega \mapsto \iii \omega\iii\) is \(K^0\)-a.s. lower semicontinuous. Consequently, a version of the Portmanteau theorem (see \cite[Example~17, p. 73]{pollard}) yields that 
		\begin{align} \label{eq: portm}
			\limsup_{n \to \infty} K^n (\{ \omega \colon \iii \omega \iii \leq C \}) \leq K^0 (\{ \omega \colon \iii \omega \iii \leq C \}).
		\end{align} 
		Recalling Lemma~\ref{lem: uni bdd comp dis}, since \(P^n\)-a.s., for all \(s < t\), 
		\begin{align*}
			| \nk \A_t - \nk \A_s | &\leq \sum_{i = \phi^n_s + 1}^{\phi^n_t}  \big| E^{P^n} \big[ \s_k (\pi_m(X_{i h_n})) - \s_k (\pi_m(X_{(i - 1)h_n})) \mid \cF^{d}_{(i - 1)h_n} \big] \big| \\&\leq (\phi^n_t - \phi^n_s - 1) \, h_n \, C
			\leq (t - (\lfloor s / h_n \rfloor + 1) h_n ) \, C
			\leq (t - s) \, C, \phantom \int
		\end{align*}
		we have \(K^n (\{ \omega \colon \iii \omega \iii \leq C \}) = 1\) and consequently, \eqref{eq: portm} implies \(K^0 (\{ \omega \colon \iii \omega \iii \leq C \}) = 1\). We conclude that \(K^0\) is supported on the space of Lipschitz continuous functions (with Lipschitz constant~\(C\)). 
		The proof is complete. 
	\end{proof}
	
Define the product space
	\(
	\Theta := \Omega^{d} \times \D^\N
	\)
	and endow it with the product topology. Moreover, set 
	\begin{align*}
		D^n := P^n \circ ({^n}\hspace{-0.05cm}Z, \nl \A, \nm \A, \dots)^{-1}, \quad n \in \N.
	\end{align*}
By the Arzel\`a--Ascoli theorem, \(\omega^n \to \omega^0\) in \(\Omega^{d}\) implies \(\omega^n (\phi^n h_n) \to \omega^0\) locally uniformly. Hence, thanks to Skorokhod's coupling theorem, we see that \(P^n \circ {^n}\hspace{-0.05cm} Z^{-1} \to P^0\). 
	Using this observation and Lemma~\ref{lem: L-tightness}, we conclude that \(\{D^n \colon n \in \mathbb{N}\}\) is relatively compact in \(\cP(\Theta)\). In particular, from the second coordinate on, the marginal distributions of all accumulation points are supported on the space of Lipschitz continuous functions. 
	
	\smallskip 
	Let \(D^0\) be an accumulation point of \(\{D^n \colon n \in \N\}\) and denote the coordinate map on \(\Theta\) by \((Y^0, Y^1, \dots)\). Notice that the first marginal of \(D^0\) coincides with \(P^0\).
	We deduce from \cite[Proposition~IX.1.4]{JS}, where hypothesis (iv) follows readily from the Arzel\`a--Ascoli theorem, that the processes
	\[
	\s_k (\pi_m (Y^0)) - Y^k, \quad k \in \N,  
	\]
	are \(D^0\)-martingales for the natural filtration generated by \(Y\). 
	Let \((\cG^0_t)_{t \geq 0}\) be the natural filtration of \(\pi_m(Y^0)\). We get from \cite[Theorem~9.19, Proposition~9.24]{jacod79} that \(\s_k (\pi_m(Y^0))\) is a \(D^0\)-\((\cG^0_t)_{t \geq 0}\)-semimartingale and that its predictable compensator is given by 
	\[
	\int_0^\cdot E^{D^0} \big[ d Y^k_s / d \llambda \mid \cG^0_{s-} \big] ds.
	\]
	Consequently, by standard rules on changes of probability spaces (see \cite[Section~10.2.a]{jacod79}), to conclude \(Q^0 \in \cD (x_0)\) it remains to show that \(D^0\)-a.s., for \(\llambda\)-a.a. \(s \in \bR_+\),
	\begin{align*}
		\big( E^{D^0} \big[ d Y^k_s / d \llambda \mid \cG^0_{s-} \big] \big)_{k = 1}^\infty \in \Psi (\pi_m(Y^0_s)).
	\end{align*}
	By Standing Assumption~\ref{SA: 3}~(ii), the set \(\Psi (x)\) is convex for every \(x \in \bR^m\).
	Using this convexity and \cite[Theorems II.4.3, II.6.2]{sion}, it suffices to prove that \(D^0\)-a.s., for \(\llambda\)-a.a. \(s \in \bR_+\), 
	\[
	d Y^k_s / d \llambda \in \Psi (\pi_m(Y^0_s)).
	\] 
	This is the program for the remainder of this proof.
	
	Thanks to Skorokhod's coupling theorem, there are realizations \((\n Y^0, \n Y^1, \n Y^2, \dots)\) of \((Y^0, Y^1, Y^2, \dots)\) under \(D^n\) such that a.s. \((\n Y^0, \n Y^1, \n Y^2, \dots) \to (\o Y^0, \o Y^1, \o Y^2, \dots)\). 
	For \((h, \omega, s) \in (0, \infty) \times \Omega^{d} \times \bR_+\), 
	define 
	\[
	\Phi (h, \omega, s) :=  \Big\{ \big( \tfrac{1}{h}A_h (\s_k \circ \pi_m) (\lambda, \omega (s)) \big)_{k = 1}^\infty \colon \lambda \in \Lambda \Big\}
	\]
	and
	\[
	\Phi (0, \omega, s) := \Psi (\pi_m(\omega (s))). 
	\]
	Take \(N \in \N\). Using once again \cite[Theorems II.4.3, II.6.2]{sion}, we get a.s., for every \(t \in \bR_+\), 
	\begin{equation*} \begin{split}
		 \big(  (\n Y^k_{t + 1 / N} - \n Y^k_{t}) / (\phi^n_{t + 1/ N} &- 1 - \phi^n_t) h_n \big)_{k = 1}^\infty 
		 \\&\in \co \, \Big( \bigcup \Big\{ \Phi (h_n, \n Y^0, s) \colon s \in [\phi^n_t h_n, (\phi^n_{t + 1/ N} - 1) h_n] \Big\} \Big), 
	\end{split}
\end{equation*}
	where \(\co\) denotes the closure of the convex hull. We now investigate how the sets on the r.h.s. behave when \(n \to \infty\).
	
	\begin{lemma} \label{lem: upper hemi}
		The set-valued map
		\begin{align*}
		(\{h_n \colon n \in \N\} \cup \{0\})\hspace{0.02cm} \times &\, \Omega^{d} \times \bR_+ \times \bR_+
		\\&\ni (h, \omega, a, b) \mapsto \co \, \Big( \bigcup \Big\{ \Phi (h, \omega, s) \colon s \in [a\wedge b, a \vee b] \Big\} \Big)
		\end{align*}
		is upper hemicontinuous and compact-valued.
	\end{lemma} 
	\begin{proof}
		First of all, by virtue of \cite[Lemma~11.2.1]{SV}, the Standing Assumptions~\ref{SA: 2} and \ref{SA: 3} imply that
		\begin{align} \label{eq: conv gen}
		\frac{1}{h_n} A_{h_n} (f) \to \langle \nabla f, \ob\, \rangle +\tfrac{1}{2} \on{tr} \big[ \nabla^2 f \, \oa\, \big], \quad f \in C^2_b (\bR^{d}; \bR), 
		\end{align}
		uniformly on compacts.
		Using this observation, we can conclude that the set-valued map \(\Phi\) is continuous by our Standing Assumptions~\ref{SA: 1} -- \ref{SA: 3} and \cite[Proposition~1.4.14]{AF}.
		Hence, by \cite[Theorem~17.23]{charalambos2013infinite}, \((h, \omega, a, b) \mapsto \bigcup\, \{ \Phi (h, \omega, s) \colon s \in [a \wedge b, a \vee b]\}\) is continuous and compact-valued. Finally, \cite[Theorems~5.35, 17.35]{charalambos2013infinite} yield the claim of the lemma. 
	\end{proof}
	As all \(\o Y^1, \o Y^2, \dots\) are a.s. continuous, we have a.s. \(\n Y^k \to \o Y^k\) locally uniformly (and not only in the Skorokhod \(J_1\) topology). 
	Now, by \cite[Theorem~17.25]{charalambos2013infinite} and Lemma~\ref{lem: upper hemi}, letting \(n \to \infty\), we conclude that a.s., for all \(t \in \bR_+\), 
	\[
	\big( N (\o Y^k_{t + 1/N} - \o Y^k_t) \big)_{k =1}^\infty \in \co \, \Big( \bigcup \Big\{ \Psi (\pi_m (\o Y^0_s)) \colon s \in [t, t + 1 / N] \Big\} \Big).
	\]
	Reusing this argument, letting \(N \to \infty\) shows that a.s., for \(\llambda\)-a.a. \(t \in \bR_+\), 
	\[
	\big( d \o Y^k_t / d \llambda \big)_{k = 1}^\infty \in \co \,\big( \Psi (\pi_m (\o Y^0_t)) \big) = \Psi (\pi_m (\o Y^0_t)),
	\]
	where we use the convexity hypothesis on \(\Psi\) for the final equality, see Standing~Assumption~\ref{SA: 3}~(ii).
	By virtue of Lemma~\ref{lem: MP chara d set}, we conclude that \(Q^0 \in \cD (x_0)\). 
\end{proof}

\begin{proof}[Second proof of Proposition~\ref{prop: upper hemi}~(ii)]
	Let \(P^n = R^n \circ \Y (h_n)^{-1} \in \mathcal{P}(\Omega^d)\) be as in the proof of part (i). 
	Recall from Section~\ref{sec: characterization 1} (more precisely, see \eqref{eq: formula cond prob}) that there exists a \(\Lambda\)-valued \((\cG_k)_{k = 0}^\infty\)-adapted process \(\f^n = (\f^n_k)_{k = 0}^\infty\) on \(\Sigma^d\) such that 
	\[
	R^n ( Y_k \in dy \mid \cG_{k - 1} ) = \Pi_{h_n} (\f^n_{k - 1}, Y_{k - 1}, dy), \quad k \in \mathbb{N}.
	\]
	As a consequence, for every \(f \in C_b (\bR^d; \bR)\), the process 
	\begin{align} \label{eq: martingale}
f (Y_k) - \sum_{i = 0}^{k - 1} A (f) (\f_i^n, Y_i), \quad k \in \mathbb{Z}_+, 
	\end{align} 
	is an \(R^n\)-martingale. Notice that \(\Y_{kh_n} (h_n) = Y_k\). Hence, we can reformulate \eqref{eq: martingale} as
	\begin{equation} \label{eq: martigale into relaxed}
		\begin{split}
		f (Y_k) &- \sum_{i = 0}^{k - 1} A (f) (\f_i^n, Y_i)
		\\&= f (\Y_{k h_n} (h_n)) - \sum_{i = 0}^{k - 1} A (f) (\f_i^n, \Y_{k h_n} (h_n)) 
		\\&= f (\Y_{k h_n} (h_n)) - \int_0^{k} A (f) (\f^n_{\lfloor z \rfloor h_n}, \Y_{\lfloor z \rfloor h_n} (h_n)) dz
		\\&= f (\Y_{k h_n} (h_n)) - \int_0^{kh_n} \tfrac{1}{h_n} \, A (f) (\f^n_{\lfloor z /h_n \rfloor h_n}, \Y_{\lfloor z /h_n \rfloor h_n} (h_n)) dz
		\\&= f (\Y_{k h_n} (h_n)) - \int_0^{kh_n} \hspace{-0.15cm} \int \tfrac{1}{h_n} \, A (f) (\lambda, \Y_{\lfloor z /h_n \rfloor h_n} (h_n)) \mathfrak{m}^n (d \lambda, dz),
	\end{split}
	\end{equation} 
	where 
	\[
	\mathfrak{m}^n (d \lambda, dz) := \delta_{\f^n_{\lfloor z / h_n \rfloor h_n}} (d \lambda) dz \in \m.
	\]
	Define \(Q^n := R^n \circ (\Y(h_n), \mathfrak{m}^n)^{-1} \in \mathcal{P} (\Omega^d \times \m)\). By the compactness of \(\m\) and the tightness of \((P^n)_{n = 1}^\infty\), which we proved in part (i), the sequence \((Q^n)_{n = 1}^\infty\) is tight in the space \(\mathcal{P}(\Omega^d \times \m)\). Possibly passing to a convergent subsequence, we may assume that \(Q^n \to Q^0\) in \(\mathcal{P}(\Omega^d \times \m)\). The aim for the remainder of this proof is to show that \(Q^0 \circ \on{pr}_{\Omega^m}^{-1} \in \cD(x^0)\). By virtue of Lemma~\ref{lem: relaxed controls}, this is implied by \(Q^0 \in \cK (x^0)\), where \(\cK (x^0)\) is the set of relaxed control rules as defined in Section~\ref{sec: characterization 2}.
	
	To show this, first notice that \(Q^0 (X_0 = (0_r, x^0)) = 1\) by the continuity of \((\omega, m) \mapsto \omega (0)\).
	Next, we prove that the processes in \eqref{eq: C martingale relax} are \(Q^0\)-martingales. 
	Take \(f \in C^2_b (\bR^d; \bR)\), two times \(s < t\) and an \(\cH_s\)-measurable function \(\mathfrak{z} \in C_b (\Omega^d \times \m; \bR)\). Using our usual notation \(\phi^n_s = \lfloor s / h_n \rfloor + 1\) and \(\phi^n_t = \lfloor t / h_n \rfloor + 1\), it follows from the \(R^n\)-martingale property of \eqref{eq: martingale} and the formula \eqref{eq: martigale into relaxed} that 
	\begin{align*}
		I_n := E^{Q^n} \Big[ \Big( f (X_{\phi^n_t h_n}) - f (X_{\phi^n_sh_n}) - \int_{\phi^n_s h_n}^{\phi^n_t h_n} \hspace{-0.15cm} \int \tfrac{1}{h_n} \, A (f) (\lambda, X_{\lfloor z /h_n \rfloor h_n}) M (d \lambda, dz) \Big) \, \mathfrak{z}\, \Big] = 0.
	\end{align*}
	With \(\bar{L} (f)\) as defined in \eqref{eq: bar L}, we get from the Arzel\`a--Ascoli theorem, \eqref{eq: conv gen} and the continuity of \(\bar{b}\) and \(\bar{a}\) (see Standing Assumption~\ref{SA: 3}~(i)) that 
	\begin{align*}
		f (X_{\phi^n_t h_b}) - f (X_{\phi^n_s h_n}) - \int_{\phi^n_s h_n}^{\phi^n_t h_n} \hspace{-0.15cm} \int &\, \tfrac{1}{h_n} \, A (f) (\lambda, X_{\lfloor z /h_n \rfloor h_n}) M (d \lambda, dz) \\&\to f (X_t) - f (X_s) - \int_s^t \hspace{-0.075cm} \int \bar{L} (f) (\lambda, X_r) M (d \lambda, dr)
	\end{align*}
	uniformly on compact subsets of \(\Omega^d \times \m\). 
	Further, by \cite[Theorem~8.10.61]{bogachev}, the map
	\[
	(\omega, m) \mapsto f (\omega (t)) - f (\omega(s)) - \int_s^t \hspace{-0.075cm} \int L (f) (\lambda, \omega (r)) m (d \lambda, dr) \, \mathfrak{z} (\omega, m)
	\]
	is bounded and continuous on \(\Omega^d \times \m\) (where we use again the continuity assumptions from Standing Assumption~\ref{SA: 3}~(i)).
	Hence, by a variant of \cite[Problem~2.4.12]{KaraShre}, where the uniform boundedness hypothesis follows from Lemma~\ref{lem: uni bdd comp dis}, we conclude that
	\begin{align*}
	E^{Q^0} \Big[ \Big( & \,f (X_t) - f (X_s) - \int_s^t \hspace{-0.075cm} \int L (f) (\lambda, X_r) M (d \lambda, dr) \Big) \, \mathfrak{z} \, \Big] 
	= \lim_{n \to \infty} I_n= 0.
	\end{align*} 
	In summary, we proved that \(Q^0 \in \cK (x^0)\), which completes the proof. 
\end{proof} 
\subsection{The ``Lower'' Part}
In this section we establish a property that is closely connected to the sequential characterization of a lower hemicontinuous set-valued map.

\begin{proposition} \label{prop: lower hemi}
	Assume that Condition~\ref{cond: main1} holds.
	Take a sequence \((h_n)_{n = 1}^\infty \subset (0, \infty)\) such that \(h_n \searrow 0\), a sequence \((x_n)_{n = 0}^\infty \subset \bR^m\) such that \(x_n \to x_0\), and a measure \(Q \in \cD (x_0)\). Then, there exists a subsequence \((h_{N_n})_{n = 1}^\infty \subset (h_n)_{n = 1}^\infty\) and a sequence \(Q^{h_{N_n}} \in \cC^{h_{N_n}} (x_{N_n})\) such that \(Q^{h_{N_n}} \to Q\) weakly. 
\end{proposition}

\begin{proof}
	Before we give a detailed proof let us shortly comment on tactics. As we will see below, any measure \(Q \in \cD(x_0)\) can be approximated by laws of strong solution processes of SDEs of the type
	\[
	d Y_t = b (\alpha_t (B), Y_t) dt + \sigma (\alpha_t (B), Y_t) d B_t, \quad Y_0 = x_0, 
	\]
	where \(\alpha\) are step processes with continuous weights and \(B\) is an \(r\)-dimensional Brownian motion. Consequently, it suffices to weakly approximate such strong solution processes by suitable chains. To conclude convergence to a specifically given law, one typically requires some sort of uniqueness property. At first glance, the above SDE seems to lack such a property. Yet, it follows easily from the Lipschitz assumptions in Condition~\ref{cond: main1} that the bi-variate process \((B, Y)\) solves an SDE that satisfies uniqueness in law. Consequently, our strategy is to approximate this bi-variant process and then to deduce the statement of the lemma through projection. We now provide the details. 
	
	\smallskip 
	Take \(Q \in \cD (x_0)\). By a version of Lemma~\ref{lem: relaxed controls}, there exists a relaxed control rule \(\overline{Q}\) (in the sense of Section~\ref{sec: characterization 2} with \((\bar{b},\bar{a})\) and \(d\) replaced by \((b, \sigma \sigma^*)\) and \(m\), respectively) such that \(Q = \overline{Q} \circ \on{pr}_{\Omega^m}^{-1}\).
By \cite[Lemma~5.7]{PTZ21}, we may approximate \(\overline{Q}\) by piecewise constant strong control rules and consequently, \(Q\) by whose projections to \(\Omega^m\).
	To explain this in detail, suppose that \((\Omega^r, \cF^r, (\cF^r_t)_{t \geq 0}, \cW)\) is the canonical setup endowed with the \(r\)-dimensional Wiener measure \(\cW\). Let \(\Gamma_{S, 0}\) be the set of all \(\Lambda\)-valued piecewise constant control processes \(\alpha = (\alpha_t)_{t \geq 0}\) defined on the space \(\Omega^r\), i.e., all \(\Lambda\)-valued processes of the form 
	\[
	\alpha = \sum_{k = 0}^\infty \alpha (k) \1_{[t_k, t_{k + 1})}, 
	\]
	where the weights \(\alpha (k)\) are \(\Gamma\)-valued, \(\cF^r_{t_k}\)-measurable and \(0 =: t_0 < t_1 < \dots < t_n \to \infty\) is a deterministic sequence. 
	Then, \cite[Lemma~5.7]{PTZ21} yields the existence of a sequence \((\alpha^n)_{n = 1}^\infty \subset \Gamma_{S, 0}\) such that the laws of 
	\begin{align} \label{eq: SDE Y}
	d Y^n_t = b (\alpha^n_t (X), Y^n_t) dt + \sigma (\alpha^n_t (X), Y^n_t) d X_t, \quad Y^n_0 = x_0, 
\end{align}
converge weakly to \(Q\). Here, we recall that the coordinate process \(X\) is an \(r\)-dimensional Brownian motion under \(\cW\) and that the SDEs \eqref{eq: SDE Y} are well-known to have strong solutions by the Lipschitz assumptions from Condition~\ref{cond: main1}.
Therefore, by a routine subsequence argument, for our purposes it suffices to weakly approximate SDEs of the type \eqref{eq: SDE Y} by suitable chains. As we explain now, we can even assume that the weights \(\alpha^n (k), \, n \in \N, \, k \in \mathbb{Z}_+\), of the piecewise controls in \eqref{eq: SDE Y} are continuous functions from \(\Omega^r\) into \(\Lambda\).  

Take \(\alpha^0 \in \Gamma_{S, 0}\) and recall that \(\Lambda\) is a compact convex subset of a strictly convex separable Banach space \(\mathbb{X}\).
As \(C_b (\Omega^r; \mathbb{X})\) is dense in \(L^1 (\Omega^r, \cF^r, \cW; \mathbb{X})\) by \cite[Proposition~1.3.9]{GP}, there exists a sequence \((\alpha^n)_{n = 1}^\infty \subset \Gamma_{S, 0}\), whose weights \(\alpha^n (k), \, n \in \N, \, k \in \mathbb{Z}_+,\) are continuous functions from \(\Omega^r\) into \(\mathbb{X}\), such that \(\alpha^n \to \alpha^0\) in the ucp topology, i.e., uniformly on compact time intervals in probability. 
Let \(\on{pr}_\Lambda\) be the metric projection from \(\mathbb{X}\) to \(\Lambda\). By \cite[Theorem~3, Corollary~7]{wul}, the map \(\on{pr}_\Lambda\) is continuous. 
Hence, because \(\on{pr}_\Lambda (\alpha^0) = \alpha^0\), applying \(\on{pr}_\Lambda\) to the approximation sequence \((\alpha^n)_{n = 1}^\infty\), we can also assume that the weights \(\alpha^n (k), \, n \in \N, \, k \in \mathbb{Z}_+\), are \(\Lambda\)-valued.
By the Lipschitz assumptions from Condition~\ref{cond: main1}, there are strong solutions \(Y^n\) to the SDEs \eqref{eq: SDE Y} and it follows from \cite[Theorem~16.4.3]{CE15} that \(Y^n \to Y^0\) in the ucp topology. In particular, the laws of \(Y^n\) converge weakly to the law of~\(Y^0\). 

Summing up, it suffices to approximate the law of a strong solution to the SDE 
	\begin{align} \label{eq: SDE Y main approx argument} 
	d Y_t = b (\alpha_t (X), Y_t) dt + \sigma (\alpha_t (X), Y_t) d X_t, \quad Y_0 = x_0, 
\end{align}
defined on \((\Omega^r, \cF^r, (\cF^r_t)_{t \geq 0}, \cW)\) where \(\alpha \in \Gamma_{S, 0}\) has continuous weights. 
This is the program for the remainder of this proof.

\smallskip 
To fix notation, we write 
\[
\alpha = \sum_{k = 0}^\infty \alpha (k) \1_{[t_k, t_{k + 1})}, 
\]
where we recall that \(\alpha (k)\) are \(\cF_{t_k}\)-measurable and continuous on \(\Omega^r\). 

For the following construction, we work on a filtered probability space, with probability measure \(P\), that supports a sequence \(U^1, U^2, \dots\)  of i.i.d. \(\on{Uni} \, ( [0, 1] )\) random variables. 
	By \cite[Lemma~4.22]{Kallenberg}, for every \(h > 0\), there exists a Borel function \(\k_h \colon \Lambda \times \bR^{d} \times [0, 1] \to \bR^{d}\) such that 
	\[
	P \circ \k_h (\lambda, x, U^1)^{-1} (dy) = \Pi_h (\lambda, x, dy), \quad \lambda \in \Lambda, \, x \in \bR^{d}.
	\]
	
	Let \(\Z \colon \Sigma^d \to \Omega^r\) be defined by 
	\begin{align*} 
		\Z_t (h) := \Big( \Big\lfloor \, \frac{t}{h}\, \Big \rfloor + 1 - \frac{t}{h} \Big)  \, \pi_r(Y_{\lfloor t/h\rfloor})+ \Big( \frac{t}{h} - \Big\lfloor \, \frac{t}{h}\, \Big \rfloor\Big) \,  \pi_r (Y_{\lfloor t/h \rfloor + 1}), \quad t \geq 0,
	\end{align*}
	where \(\pi_r \colon \bR^d \to \bR^r\) is given by \(\pi_r (x_1, \dots, x_d) := (x_1, \dots, x_r)\). 
	
	Define a sequence \(Z^{h_n}_0, Z^{h_n}_1, Z^{h_n}_2, \dots\) of \(\bR^d\)-valued random variables as follows:
	\begin{align*}
	Z^{h_n}_0 &:= (0_r, x_n), \\ Z^{h_n}_k &:= \k_{h_n} (\alpha_{m - 1} (\Z_{\cdot \wedge (k - 1) h_n} (h_n) \circ Z^{h_n}), Z^{h_n}_{k - 1}, U^k), \quad t_{m - 1} \leq kh_n < t_{m},
	\end{align*}
	where \(Z^{h_n} := (Z^{h_n}_0, Z^{h_n}_1, \dots)\). 
	
	Finally, define 
	\[
	Q^{h_n} := Q \circ (\Y (h_n) \circ Z^{h_n})^{-1}, \quad n \in \N, 
	\]
	which is the law of the linear interpolation (recall \eqref{eq: interpol complete} for the definition of \(\Y (h)\)) of the chain \(Z^{h_n}\). 
	It is readily seen that \(Q^{h_n} \circ \on{pr}_{\Omega^m}^{-1} \in \cC^{h_n} (x_n)\).
		
		W.l.o.g. we assume that \(h_1 < \theta\) with \(\theta \in (0, \delta)\) such that \eqref{eq: theta bounded} holds.
	Relative compactness of \(\{ Q^{h_n} \colon n \in \N\}\) in the space \(\cP(\Omega^d)\) follows as in the proof of Proposition~\ref{prop: upper hemi}~(i). To ease our presentation, we assume that the sequence \((Q^{h_n})_{n = 1}^\infty\) converges weakly to a probability measure \(Q^0\).
	
We now use a martingale problem argument to identify \(Q^0\) and relate it to \eqref{eq: SDE Y main approx argument}. By the continuity of \(\omega \mapsto \omega (0)\), it is evident that \(Q^0 \circ X^{-1}_0 = \delta_{(0, x_0)}\).
	Take two times \(s < t\), \(f \in C^2_b (\bR^d; \bR)\) and let \(g \in C_b (\Omega^d; \bR)\) be \(\cF^{d}_s\)-measurable. 
	By construction, it follows that
	\begin{align*}
		I_n :\hspace{-0.1cm}&=	E^{Q^{h_n}} \Big[ \Big( f (X_{\phi^n_t h_n}) - f (X_{\phi^n_s h_n}) 
		\\&\hspace{1.75cm}- \sum_{i = \phi^n_s}^{\phi^n_t - 1} \sum_{m = 1}^{\infty} A_{h_n} (f) (\alpha_{m - 1} (\pi_r (X_{\cdot \wedge (i - 1)h_n})), X_{ih_n}) \1_{\{t_{m - 1} \leq i h_n < t_m\}} \Big) \, g (X) \, \Big] \\&= 0,
	\end{align*}
	where \(\phi^n_s = \lfloor s / h_n \rfloor + 1\) and \(\phi^n_t = \lfloor t / h_n \rfloor + 1\).
	With \(\bar{L}\) defined as in \eqref{eq: bar L}, notice that 
	\[
	\omega \mapsto \Big( f (\omega (t)) - f (\omega (s)) - \int_s^t \bar{L} (f) (\alpha_z (\pi_r (\omega)), \omega(z)) dz \Big) \, g (\omega)
	\]
	is bounded and continuous by Standing~Assumption~\ref{SA: 3}~(i). 
	\begin{lemma}
The sequence
	\begin{align} \label{eq: sequence approx MP}
		 \sum_{i = \phi^n_s}^{\phi^n_t - 1}  \sum_{m = 1}^\infty  A_{h_n} (f) (\alpha_{m - 1} (\pi_r (X_{\cdot \wedge (i - 1)h_n})), X_{ih_n}) \1_{\{t_{m - 1} \leq i h_n < t_m\}}, \quad n \in \N, 
	\end{align}
	converges to 
	\begin{align*}
		\int_s^t L (f) (\alpha_z (\pi_r (X)), X_z) dz
	\end{align*}
	uniformly on compact subsets of \(\Omega^d\).
\end{lemma}
\begin{proof}
	We assume that for some \(N \in \N\) that \(t_{N - 1} \leq s < t \leq t_{N}\). The general result follows from pasting. 
	It suffices to investigate the convergence of
	\begin{align*}
	\sum_{j = \phi^n_s}^{\phi^n_t - 1} A_{h_n} (f) &(\alpha_{N - 1} (\pi_r (X_{\cdot \wedge (i - 1)h_n})), X_{j h_n}) 
	\\&= \int_{\phi^n_s h_n}^{\phi^n_t h_n} \frac{A_{h_n} (f) (\alpha_{N - 1} (\pi_r (X_{\cdot \wedge ( \lfloor z / h_n \rfloor h_n - 1) h_n})), X_{\lfloor z / h_n \rfloor h_n})}{h_n} dz.
	\end{align*} 
	It follows readily from \eqref{eq: conv gen}, the continuity of \(\ob\) and \(\oa\) (see Standing Assumption~\ref{SA: 3}~(i)), and the Arzel\`a--Ascoli theorem, that the final integral converges to 
	\[
	\int_s^t \bar{L} (f) (\alpha_{N - 1} (\pi_r (X)), X_z) dz
	\]
	uniformly on compact subsets of \(\Omega^d\). 
\end{proof}
	
	Notice that the sequence \eqref{eq: sequence approx MP} is uniformly bounded by Lemma~\ref{lem: uni bdd comp dis}.
	Hence, by \cite[Problem~2.4.12]{KaraShre}, the previous lemma and the weak convergence \(Q^{h_n}\to Q^0\), we get that
	\[
E^{Q^0} \Big[ \Big( f (X_t) - f (X_s) - \int_s^t \bar{L} (f) (\alpha_z (\pi_r(X)), X_z) dz \Big) g (X) \, \Big] = \lim_{n \to \infty} I_n = 0.
	\] 
	Finally, by standard results on martingale problems (cf., e.g., \cite[Theorem~2.10]{jacod1981weak} or \cite[Theorem~VII.2.7]{RY}), under \(Q^0\), the process \(\pi_r (X)\) is a standard Brownian motion and 
	\[
	d \pi_m (X_t) = b (\alpha_t (\pi_r (X)), \pi_m (X_t)) dt + \sigma (\alpha_t (\pi_r (X)), \pi_m (X_t)) d \pi_r (X_t), \quad \pi_m (X_0) = x_0.
	\]
	Now, as the Lipschitz assumptions from Condition~\ref{cond: main1} entail uniqueness in law for this SDE, we conclude that \(Q^0 \circ \on{pr}_{\Omega^m}^{-1} = Q\). 
In summary, we constructed an appropriate approximation and, by the considerations above, the proof is complete.
	\end{proof}

\subsection{Proof of Theorem~\ref{theo: main1}~(ii)} 

We are in the position to give a proof for the Hausdorff metric convergence in Theorem~\ref{theo: main1}~(ii). 

\smallskip
Recall that we have to prove that 
\begin{align*}
\mathsf{h} (\cC^{h_n} (x_n), \cD (x_0)) = \max \Big\{ \sup_{P \in \cC^{h_n} (x_n)} \mathsf{p} (P, \cD (x_0)), \, \sup_{Q \in \cD (x_0)} \mathsf{p} (Q, \cC^{h_n} (x_n)) \Big\} \to 0.
\end{align*}

		We start investigating the first term on the right hand side. 	
	By the definition of the supremum, for every \(n \in \mathbb{N}\), there exists a measure \(P^n \in \cC^{h_n} (x_n)\) such that 
	\[
	\sup_{P \in \cC^{h_n} (x_n)} \mathsf{p} (P, \cD (x_0)) \leq \mathsf{p} (P^n, \cD (x_0)) + \frac{1}{n}. 
	\]
	Hence, we notice that 
	\[
	\limsup_{n \to \infty} \sup_{P \in \cC^{h_n} (x_n)} \mathsf{p} (P, \cD (x_0)) = \limsup_{n \to \infty} \mathsf{p} (P^n, \cD (x_0)).
	\]
	Up to passing to a subsequence, that we suppress in our notation for simplicity, we may assume that
	\[
	\lim_{n \to \infty} \mathsf{p} (P^n, \cD (x_0))
	\]
	exists. 
	By Proposition~\ref{prop: upper hemi}, the family \(\{P^n \colon n \in \mathbb{N}\}\) is relatively compact in \(\cP(\Omega^m)\) and any of its accumulation points is contained in \(\cD (x_0)\). Again, up to passing to a subsequence, we may assume that \(P^n \to P \in \cD (x_0)\) weakly. By the continuity of the distance function, it follows that 
	\[
	\lim_{n \to \infty} \mathsf{p} (P^n, \cD (x_0)) = \mathsf{p} (P, \cD (x_0)) = 0.
	\]
	This shows that the first term in the definition of \(\mathsf{h}\) vanishes. 
	
	\smallskip
	Next, we turn to the second term. Again by definition of the supremum, there exists a sequence \((Q^n)_{n = 1}^\infty \subset \cD (x_0)\) such that 
	\[
	\limsup_{n \to \infty} \sup_{Q \in \cD (x_0)} \mathsf{p} (Q, \cC^{h_n} (x_n)) = \limsup_{n \to \infty} \mathsf{p} (Q^n, \cC^{h_n} (x_n)).
	\]
	Up to passing to a subsequence, we presume that the limit 
	\[
	\lim_{n \to \infty} \mathsf{p} (Q^n, \cC^{h_n} (x_n))
	\]
	exists. 
	By the compactness of \(\cD (x_0)\), which follows from Theorem~\ref{theo: main1}~(i), the family \(\{Q^n \colon n \in \mathbb{N}\}\) is relatively compact in \(\cP (\Omega^m)\) and any of its accumulation points is contained in \(\cD (x_0)\). Up to passing to a subsequence, we can assume that \(Q^n \to Q^0 \in \cD (x_0)\) weakly. 
	
	We deduce from Proposition~\ref{prop: lower hemi}
	that there exists a subsequence \((h_{N_m})_{m = 1}^\infty\) of the sequence \((h_m)_{m = 1}^\infty\) and measures \(P^{m} \in \cC^{h_{N_m}} (x_{N_m})\) such that \(P^{m} \to Q^0\) as \(m \to \infty\). 
	Now, we conclude that 
	\begin{align*}
	\mathsf{p} (Q^n, \cC^{h_{N_n}} (x_{N_n})) \leq \mathsf{p} (Q^n, P^{n}) \leq \mathsf{p} (Q^n, Q^0) + \mathsf{p} (Q^0, P^{n}) \to 0.
	\end{align*}
	This proves that 
	\[
\lim_{n \to \infty} \mathsf{p} (Q^n, \cC^{h_n} (x_n)) = 0, 
\]
and all together, \(\mathsf{h} (\cC^{h_n} (x_n), \cD (x_0)) \to 0\). The proof is complete. \qed

\section{Proofs of Theorems~\ref{theo: upper} and \ref{theo: lower}} \label{sec: pf iii}

In this final section, we present the proofs for the Theorems~\ref{theo: upper} and \ref{theo: lower}.

\begin{proof}[Proof of Theorem~\ref{theo: upper}]
	There exists a sequence \((P^n)_{n = 1}^\infty\) such that \(P^n \in \cC^{h_n} (x_n)\) and
	\[
	\limsup_{n \to \infty} \sup_{P \in \cC^{h_n} (x_n)} E^P \big[ \varphi^n \big] = \limsup_{n \to \infty} E^{P_n} \big[ \varphi^n \big]. 
	\]
	Up to passing to a subsequence, we can assume that 
	\[
	\lim_{n \to \infty} E^{P^{n}} \big[ \varphi^n \big]
	\]
	exists.
	By Proposition~\ref{prop: upper hemi}, passing possibly again to a subsequence, we may assume that \(P^n \to P^0 \in \cD (x_0)\).
	Take \(\varepsilon > 0\). Then, there exists a compact set \(K = K_\varepsilon \subset \Omega^m\) such that 
	\[
	\sup_{n \in \mathbb{N}} P^n (K^c) \leq \varepsilon. 
	\]
	As \(\varphi^n \to \varphi\) b-uc, we have 
	\[
	E^{P_n} \big[ \big| \varphi^n - \varphi \big| \big] \leq \sup_{\omega \in K} | \varphi^n (\omega) - \varphi (\omega) | + 2 \varepsilon \sup_{m \in \mathbb{Z}_+}\| \varphi^m \|_\infty  \to 2 \varepsilon \sup_{m \in \mathbb{Z}_+} \| \varphi^m \|_\infty.
	\]
Since \(\varepsilon > 0\) is arbitrary, we get that 
\[
\lim_{n \to \infty} E^{P^n} \big[ \varphi^n \big] = \lim_{n \to \infty} E^{P^n} \big[ \varphi \big].
\]
Finally, by the Portmanteau theorem and the \(P^0\)-a.s. upper semicontinuity of \(\varphi\),
\begin{align*}
	\lim_{n \to \infty} E^{P^{n}} \big[ \varphi^n\big] \leq E^{P^0} \big[ \varphi \big] \leq \sup_{P \in \cD (x_0)} E^P \big[ \varphi \big]. 
\end{align*}
	In summary, 
	\begin{align*}
		\limsup_{n \to \infty} \sup_{P \in \cC^{h_n} (x_n)} E^P \big[ \varphi^n\big] \leq \sup_{P \in \cD (x_0)} E^P \big[ \varphi\big].
	\end{align*}
	The proof is complete.
\end{proof}

\begin{proof}[Proof of Theorem~\ref{theo: lower}]
	Up to passing to a subsequence, we can assume that 
	\[
	\lim_{n \to \infty} \sup_{P \in \cC^{h_n} (x_n)} E^P \big[ \varphi^n \big]
	\]
	exists. 
	Take an arbitrary \(Q \in \cD (x_0)\). Then, by Proposition~\ref{prop: lower hemi}, there exists a sequence \((P^{n})_{n = 1}^\infty\) such that \(P^{n} \in \cC^{h_{N_n}} (x_{N_n})\) and \(P^{n} \to Q\). As in the proof of Theorem~\ref{theo: upper}, we observe that 
	\[
	\liminf_{n \to \infty} E^{P^n} \big[ \varphi^n \big] = \liminf_{n \to \infty} E^{P_n} \big[ \varphi \big]. 
	\]
	Now, by the Portmanteau theorem and the \(Q\)-a.s. lower semicontinuity of \(\varphi\), we get that 
	\begin{align*}
		E^{Q} \big[ \varphi  \big]  &\leq \liminf_{n \to \infty} E^{P^{n}} \big[ \varphi \big] \phantom {\sup_{P \in \cC^{h_n} (x_n)}}
		\\& =	\liminf_{n \to \infty} E^{P^n} \big[ \varphi^n \big] \phantom {\sup_{P \in \cC^{h_n} (x_n)}}
		\\&\leq \liminf_{n \to \infty} \sup_{P \in \cC^{h_{N_n}} (x_{N_n})} E^P \big[ \varphi^n \big]
		\\&= \lim_{n \to \infty} \sup_{P \in \cC^{h_n} (x_n)} E^P \big[ \varphi^n \big].
	\end{align*}
	As \(Q \in \cD (x_0)\) was arbitrary,
	\[
	\sup_{Q \in \cD (x_0)} E^{Q} \big[ \varphi  \big] \leq \lim_{n \to \infty} \sup_{P \in \cC^{h_n} (x_n)} E^P \big[ \varphi^n \big],
	\]
	and the proof is complete.
	\end{proof}

\end{document}